\documentclass[reqno,openany,dvipdfmx]{amsart}
\usepackage{amsthm,amsmath,amssymb,amsthm}
\usepackage[left=3.5cm, right=3.5cm, includefoot]{geometry}
\usepackage[all]{xy}
\usepackage{color}
\usepackage{appendix}
\usepackage{tikz}

\newcommand{\fancyot}{\mathbin{\text{\footnotesize\textcircled{\tiny \sf T}}}}

\theoremstyle{definition}
\newtheorem{theorem}{Theorem}[section]
\newtheorem{proposition}{Proposition}[section]
\newtheorem{lemma}{Lemma}[section]
\newtheorem{corollary}{Corollary}[section]

\newtheorem{definition}{Definition}[section]
\newtheorem{remark}{Remark}[section]

\numberwithin{equation}{section}

\begin{document}
\title{Inductive limits of compact quantum groups and their unitary representations}
\author[R. Sato]{Ryosuke SATO}
\address{Graduate School of Mathematics, Nagoya University, Chikusaku, Nagoya 464-8602, Japan}
\email{d19001r@math.nagoya-u.ac.jp}
\maketitle

\begin{abstract}
We will introduce the notion of inductive limits of compact quantum groups as $W^*$-bialgebras equipped with some additional structures. We also formulate their unitary representation theories. Those give a more explicit representation-theoretic meaning to our previous study of quantized characters associated with a given inductive system of compact quantum groups. As a byproduct, we will give an explicit representation-theoretic interpretation to some transformations that play an important role in the analysis of $q$-central probability measures on the paths in the Gelfand--Tsetlin graph. 
\end{abstract}

\allowdisplaybreaks{
\section{Introduction}
In our previous paper \cite{Sato1}, we introduced the notion of quantized characters associated with a given inductive system, say $\mathbb{G} = (G_N)_{N=0}^\infty$, of compact quantum groups. This work was based on the so-called Stratila--Voiculescu AF-algebra $\mathfrak{A}(\mathbb{G})$, which we regard as the group $C^*$-algebra associated with the conjectural ``inductive limit quantum group $G_\infty$" arising from $\mathbb{G}$. In the same paper \cite{Sato1}, we also gave a serious investigation of quantized characters when the inductive system $\mathbb{G}$ consists of the quantum unitary groups $U_q(N)$ with $0 < q < 1$. One of the consequences there was to give an explicit interpretation to Gorin's asymptotic analysis \cite{Gorin12} of the so-called $q$-Gelfand--Tsetlin graph in terms of $q$-deformed quantum groups. However, we have not yet constructed any explicit ``inductive limit quantum group $G_\infty$" from $\mathbb{G}$, which should fit into our theory of quantized characters as well as Gorin's analysis of $q$-Gelfand--Tsetlin graph (when $G_N = U_q(N)$). Here we need a new framework of operator algebraic quantum groups because the infinite-dimensional unitary group $U(\infty)=\varinjlim_NU(N)$ is not locally compact.

In this paper, we will introduce a natural inductive limit of the given $\mathbb{G}$ as a $W^*$-bialgebra. More precisely, we will show that Takeda's $W^*$-inductive limit (see \cite{Takeda}), denoted by $W^*(G_\infty)$ here, of compact quantum group $W^*$-algebras $W^*(G_N)$ (in the sense of Yamagami \cite{Yamagami}) admits a natural quantum group structure similar to Woronowicz algebras introduced by Masuda--Nakagami \cite{MasudaNakagami}, that is, $W^*(G_\infty)$ has a comultiplication $\hat\delta_\infty$, a unitary antipode $\hat R_\infty$ and a deformation automorphism group $\{\hat\tau^\infty_t\}_{t\in\mathbb{R}}$ on $W^*(G_\infty)$. Takeda's construction of $W^*$-inductive limits admits that $\mathfrak{A}(\mathbb{G})$ sits in $W^*(G_\infty)$ as a $\sigma$-weakly dense $*$-subalgebra. More importantly, all the normal KMS states on $W^*(G_\infty)$ with respect to $\{\hat{\tau}^\infty_t\}_{t\in\mathbb{R}}$ are naturally identified (as restriction/extension) with our previous quantized characters defined on $\mathfrak{A}(\mathbb{G})$. Thus, we can reformulate our previous theory of quantized characters associated with $\mathbb{G}$ in terms of this $W^*$-bialgebra $G_\infty=(W^*(G_\infty),\hat\delta_\infty,\hat{R}_\infty, \hat{\tau}^\infty_t)$ except for the ``inductive limit topology" on the convex set of quantized characters.  In fact, $W^*(G_\infty)$ is too large and needs a smaller subalgebra like $\mathfrak{A}(\mathbb{G})$ to introduce a natural topology on the quantized characters; see Remark \ref{rem:waidona}. 

The introduction of this explicit $W^*$-bialgebra $G_\infty$ allows us to give a natural unitary representation theory of $G_\infty$ rather than $\mathbb{G}$. For instance, we can obtain a new quantized character $\chi$ from two given ones $\chi_1, \chi_2$ through the tensor product of the unitary representations arising from those $\chi_1, \chi_2$. We will investigate this $\chi$ when $G_N = U_q(N)$ and describe it in terms of the so-called $q$-Schur generating functions associated with $\chi_1, \chi_2$. Those $q$-Schur generating functions are nothing less than the ``restrictions" of $\chi_1, \chi_2$  to the ``maximal tori". In this context, we will pay our attention to the quantum analog of the determinant representation of $U(\infty)$, which we call the quantum determinant representation. As a consequence, we will see that the tensor product of the quantum determinant representation and the unitary representation arising from any extreme quantized character is also associated with an extreme quantized character. Moreover, we will describe such a tensor product representation in terms of concrete transformations on the set of extreme quantized characters. Here, we would like to emphasize that those transformations often appear in the analysis of $q$-central probability measures on the paths in the Gelfand--Tsetlin graph (see \cite{Gorin12}, \cite{Cuenca}), and hence the present work gives, as a byproduct,  an explicit representation-theoretic interpretation to some important transformations in the analysis of $q$-central probability measures due to Gorin.

In Section 2, we review necessary facts on compact quantum groups. Section 3 is a main part in this paper, where we will introduce the notion of inductive limits of compact quantum groups. In Section 4, we investigate tensor products of unitary representations of the infinite-dimensional quantum unitary group $U_q(\infty)$, which is the inductive limit of quantum unitary groups $U_q(N)$. As an appendix, we give a possible formulation of spherical representations and spherical functions in our setting of quantum group $W^*$-algebras following Olshanski's idea, e.g., \cite{BorodinOlshanski}, \cite{Olshanski90}, \cite{Olshanski03}.

%%%%%%%%%%%%%%%%%%%%%%%%%%%%%%%%%%%%%%%%%%%%%%%%%%%%%%%%%%%%%%%%%%%%%%%%%%%%%%%%%%%%%%%%%%%%%%%%%%%%%%%%%%%%%%%%
\section{Preliminaries}\label{sec:cqg}
We review necessary facts on compact quantum groups and compact quantum group $W^*$-algebras. %Our references are \cite{NeshveyevTuset}, \cite{PodlesWoronowicz} and \cite{Tomatsu}. 

Let $G=(A(G),\delta_G)$ be a compact quantum group, that is, $A(G)$ is a unital $C^*$-algebra and $\delta_G\colon A(G)\to A(G)\otimes A(G)$ is a unital $*$-homomorphism satisfying
\begin{itemize}
\item $(\delta_G\otimes\mathrm{id})\circ\delta_G=(\mathrm{id}\otimes\delta_G)\circ\delta_G$ as $*$-homomorphisms from $A(G)$ to $A(G)\otimes A(G)\otimes A(G)$,
\item $(A(G)\otimes 1)\delta_G(A(G)),\,(1\otimes A(G))\delta_G(A(G))\subset A(G)\otimes A(G)$ are dense,
\end{itemize}
where the symbol $\otimes$ denotes the operation of minimal tensor products of $C^*$-algebras. Let $h_G$ be the Haar state of $G$, which is a quantum analog of integration by a Haar probability measure. Moreover, $\{f^G_z\}_{z\in\mathbb{C}}$ denotes the so-called Woronowicz characters of $G$. See \cite{NeshveyevTuset} for more details.

In this paper, we always assume that $\widehat{G}$ is countable, where $\widehat{G}$ is the set of all unitary equivalence classes of irreducible unitary representations of $G$. For each $\alpha\in\widehat{G}$ we fix a representative $U_\alpha\in B(\mathcal{H}_{U_\alpha})\otimes A(G)$, where the representation space $\mathcal{H}_{U_\alpha}$ must be finite dimensional. Then the matrix $F_{U_\alpha}:=(\mathrm{id}\otimes f^G_1)(U_\alpha)\in B(\mathcal{H}_{U_\alpha})$ is positive and invertible with $\mathrm{Tr}(F_{U_\alpha})=\mathrm{Tr}(F_{U_\alpha}^{-1})$, where this trace is called the \emph{quantum dimension} of $U_\alpha$ (and of $\alpha$), denoted by $\dim_q(\alpha)$. 

We denote by $\mathcal{A}(G)$ the linear subspace of $A(G)$ generated by all matrix coefficients of finite dimensional representations of $G$. Then $\mathcal{A}(G)$ becomes a $*$-subalgebra of $A(G)$. In this paper, we always assume that $A(G)$ is the universal $C^*$-algebra generated by $\mathcal{A}(G)$. The linear dual $\mathcal{A}(G)^*$ also becomes a $*$-algebra and $\mathcal{A}(G)^*\cong\prod_{\alpha\in\widehat{G}}B(\mathcal{H}_{U_\alpha})$ naturally as $*$-algebras. Then we define the three $*$-subalgebras $\mathbb{C}[G]$, $C^*(G)$ and $W^*(G)$ of $\mathcal{A}(G)^*$ that are $*$-isomorphic to 
\[\mathrm{alg}\mathchar`-\bigoplus_{\alpha\in\widehat{G}}B(\mathcal{H}_{U_\alpha}):=\left\{(x_\alpha)_{\alpha\in\widehat{G}}\in\prod_{\alpha\in\widehat{G}}B(\mathcal{H}_{U_\alpha})\,\middle|\,x_\alpha=0\text{ without finitely many }\alpha\in\widehat{G}\right\},\]
\[c_0\mathchar`-\bigoplus_{\alpha\in\widehat{G}}B(\mathcal{H}_{U_\alpha}):=\left\{(x_\alpha)_{\alpha\in\widehat{G}}\in\prod_{\alpha\in\widehat{G}}B(\mathcal{H}_{U_\alpha})\,\middle|\,\lim_{\alpha\in\widehat{G}}\|x_\alpha\|=0\right\},\]
\[\ell^\infty\mathchar`-\bigoplus_{\alpha\in\widehat{G}}B(\mathcal{H}_{U_\alpha}):=\left\{(x_\alpha)_{\alpha\in\widehat{G}}\in\prod_{\alpha\in\widehat{G}}B(\mathcal{H}_{U_\alpha})\,\middle|\,\sup_{\alpha\in\widehat{G}}\|x_\alpha\|<\infty\right\}\]
by the above $*$-isomorphism, respectively. Then $C^*(G)$ becomes a $C^*$-algebra called the \emph{group $C^*$-algebra} of $G$, and $W^*(G)$ becomes a von Neumann algebra called the \emph{group von Neumann algebra} of $G$. In what follows, we identify three $*$-algebras $\mathbb{C}[G]\subset C^*(G)\subset W^*(G)$ with 
\[\mathrm{alg}\mathchar`-\bigoplus_{\alpha\in\widehat{G}}B(\mathcal{H}_{U_\alpha})\subset c_0\mathchar`-\bigoplus_{\alpha\in\widehat{G}}B(\mathcal{H}_{U_\alpha})\subset \ell^\infty\mathchar`-\bigoplus_{\alpha\in\widehat{G}}B(\mathcal{H}_{U_\alpha}).\]

Let $\{\tau^G_t\}_{t\in\mathbb{R}}$ be the scaling group of $G$, which is a one-parameter automorphism group on $\mathcal{A}(G)$ (see \cite[Section 1.7]{NeshveyevTuset}). It is known that their dual maps $\{\hat\tau^G_t\}_{t\in\mathbb{R}}$ (defined by $\hat\tau^G_t(f):=f\circ\tau^G_t$ for $f\in\mathcal{A}(G)^*$) is given as $\{\prod_{\alpha\in\widehat{G}}\mathrm{Ad}F_{U_\alpha}^{\mathrm{i}t}\}_{t\in\mathbb{R}}$. Thus, $\{\hat\tau^G_t\}_{t\in\mathbb{R}}$ preserves $\mathbb{C}[G], C^*(G)$ and $W^*(G)$ respectively, and forms a one-parameter automorphism group. In what follows, we denote by the same symbol $\{\hat\tau^G_t\}_{t\in\mathbb{R}}$ the restrictions to $W^*(G)$ and $C^*(G)$.

Let $(\pi_G,L^2(G),\xi_G)$ be the GNS-triple associated with the Haar state $h_G$. Then, using the orthogonality relations of matrix coefficients of irreducible representations (see \cite[Theorem 1.4.3]{NeshveyevTuset}), we can show that $L^2(G)$ is isometrically isomorphic to $\ell^2\mathchar`-\bigoplus_{\alpha\in\widehat{G}}B(\mathcal{H}_{U_\alpha})$,
where each $B(\mathcal{H}_{U_\alpha})$ has the inner product given by $\langle X,Y\rangle:=\mathrm{Tr}_{\mathcal{H}_{U_\alpha}}(F_{U_\alpha} Y^*X)/\dim_q(\alpha)$. Thus, we can regard operators in $C^*(G)$ and $W^*(G)$, which act on $\ell^2\mathchar`-\bigoplus_{\alpha\in\widehat{G}}B(\mathcal{H}_{U_\alpha})$ by the left multiplication, as bounded operators on $L^2(G)$. Moreover, we always assume that $h_G$ is faithful. Thus, $A(G)$ is faithfully embedded into $B(L^2(G))$.

We define $U^G:=\bigoplus_{\alpha\in\widehat{G}}U_\alpha\in \bigoplus_{\alpha\in\widehat{G}}W^*(G)\otimes B(\mathcal{H}_\alpha)= W^*(G)\otimes B(L^2(G))$. By \cite[Theorem 3.1]{PodlesWoronowicz} and \cite[Section 2.2]{Tomatsu}, there exists a unique unital $*$-homomorphism $\hat\delta_G\colon W^*(G)\to W^*(G)\bar\otimes W^*(G)$ such that
\begin{equation}\label{eq:coproduct}
(\hat\delta_G\otimes\mathrm{id})\circ\hat\delta_G=(\mathrm{id}\otimes\hat\delta_G)\circ\hat\delta_G,\quad
(\hat\delta_G\otimes\mathrm{id})(U^G)=U^G_{23}U^G_{13},
% as *-homomorphism from $C^*(G)$ to $M(C^*(G)\otimes C^*(G)\otimes C^*(G))$
\end{equation}
\begin{equation}\label{eq:comm_action}
\hat\delta_G\circ\hat\tau^G_t=(\hat\tau^G_t\otimes\hat\tau^G_t)\circ\hat\delta_G.
\end{equation}
We denote by $\hat\epsilon_G$ and $\hat{R}_G$ the counit and the unitary antipode of $W^*(G)$, respectively. See \cite[Section 1.6]{NeshveyevTuset} for definitions. The counit $\hat\epsilon_G\colon W^*(G)\to\mathbb{C}$ is a normal $*$-homomorphism and the unitary antipode $\hat R_G\colon W^*(G)\to W^*(G)$ is an involutive normal $*$-anti-automorphism. It is known that 
\begin{equation}\label{eq:nogi}
\hat\tau^G_t\circ\hat R_G=\hat R_G\circ\hat\tau^G_t,\quad \hat\delta_G\circ\hat R_G=(\hat R_G\otimes\hat R_G)\circ\hat\delta_G^\mathrm{op},
\end{equation}
where $\hat\delta^\mathrm{op}:=\sigma\circ\hat\delta$ and $\sigma\colon W^*(G)\bar\otimes W^*(G)\to W^*(G)\bar\otimes W^*(G)$ is the flip map. Remark that $\hat R_G\circ\hat\tau^G_{-\mathrm{i}/2}=\hat\tau^G_{-\mathrm{i}/2}\circ\hat R_G$ becomes an antipode, where $\hat\tau_{-\mathrm{i}/2}$ is the analytic continuation of $\{\hat\tau_t\}_{t\in\mathbb{R}}$. More precisely, for any $x\in\mathbb{C}[G]$ 
\[m\circ(\mathrm{id}\otimes(\hat R_G\circ\hat\tau^G_{-\mathrm{i}/2}))(\hat\delta_G(x))=\hat\epsilon_G(x)1=m\circ((\hat R_G\circ\hat\tau^G_{-\mathrm{i}/2})\otimes\mathrm{id})(\hat\delta_G(x)),\]
where $m\colon W^*(G)\odot W^*(G)\to W^*(G)$ is the multiplication. Then, following Yamagami \cite{Yamagami}, we call the $(W^*(G),\hat\delta_G,\hat\epsilon_G,\hat R_G,\{\hat\tau^G_t\}_{t\in\mathbb{R}})$ a \emph{compact quantum group $W^*$-algebra}. In what follows, we use the same symbol $G$ to denote a compact quantum group $W^*$-algebra, i.e., $G=(W^*(G),\hat\delta_G,\hat\epsilon_G,\hat R_G,\{\hat\tau^G_t\}_{t\in\mathbb{R}})$.

\begin{definition}
A normal $\hat\tau^G$-KMS state on $W^*(G)$ with the inverse temperature $-1$ is called a \emph{quantized character} of $G$. We denote by $\mathrm{Char}(G)$ the set of quantized characters of $G$. For any $\alpha\in\widehat{G}$ the $\chi^\alpha\in\mathrm{Char}(G)$ defined by 
\[\chi^\alpha((x_{\alpha'})_{\alpha'\in\widehat{G}}):=\frac{\mathrm{Tr}(F_{U_\alpha} x_\alpha)}{\dim_q(\alpha)}\]
is called the \emph{indecomposable quantized character} labeled with $\alpha\in\widehat{G}$.
\end{definition}

\begin{remark}
Our previous definition of quantized characters (\cite[Definition 2.1]{Sato1}) is equivalent to the above definition of quantized characters. See Proposition \ref{prop:reason} just below. If $\{\hat\tau^G\}$ is trivial (i.e., $G$ is of Kac type), quantized characters are nothing less than normal tracial states on $W^*(G)$. If $G$ is a compact group, then each normal tratical state on $W^*(G)$ corresponds to each characters of $G$, where a character of $G$ means a positive-definite continuous function $f$ on $G$ satisfying that $f(gh)=f(hg)$ for any $g,h\in G$ and $f(e)=1$.
\end{remark}

We equip $\mathrm{Char}(G)$ with the weak${}^*$ topology induced from $W^*(G)$. Let $\mathrm{KMS}(C^*(G))$ be the set of $\hat\tau^G$-KMS state with the inverse temperature $-1$ on $C^*(G)$, equipped with the weak${}^*$ topology induced from $C^*(G)$. Remark that $\mathrm{Char}(G)$ and  $\mathrm{KMS}(C^*(G))$ are $w^*$-metrizable since $W^*(G)$ and $C^*(G)$ are separable.
\begin{proposition}\label{prop:reason}
The mapping $\chi\in\mathrm{Char}(G)\mapsto\chi|_{C^*(G)}\in\mathrm{KMS}(C^*(G))$ is an affine homeomorphism. In particular, a sequence $\chi_n\in\mathrm{Char}(G)$ converges to $\chi\in\mathrm{Char}(G)$ as $n\to\infty$ if and only if $\chi_n(x)\to\chi(x)$ as $n\to\infty$ for any $x\in C^*(G)$.
\end{proposition}
\begin{proof}
By the similar way as \cite[Lemma 2.2]{Sato1}, we can show that any $\chi\in\mathrm{Char}(G)$ can be uniquely decomposed as $\chi=\sum_{\alpha\in\widehat{G}}c_\alpha\chi^\alpha$, where the $c_\alpha$ are nonnegative coefficients satisfying $\sum_{\alpha\in\widehat{G}}c_\alpha=1$. Thus, there exists an affine bijection between $\mathrm{Char}(G)$ and $\mathrm{KMS}(C^*(G))$. In the rest of proof, we show that a sequence $\chi_n\in\mathrm{Char}(G)$ converges to $\chi\in\mathrm{Char}(G)$ if $\chi_n(x)\to\chi(x)$ as $n\to\infty$ for any $x\in C^*(G)$. We denote by 
$\chi_n=\sum_{\alpha}c_\alpha^{(n)}\chi^\alpha$ and $\chi=\sum_{\alpha\in\widehat{G}}c_\alpha\chi^\alpha$
the unique decompositions of $\chi_n$, $\chi$, respectively. By the assumption, $c_\alpha^{(n)}\to c_\alpha$ as $n\to\infty$ for any $\alpha\in\widehat{G}$. Thus, by the dominated convergence theorem with respect to the counting measure on $\widehat{G}$, the sequence $\chi_n(x)$ converges to $\chi(x)$ for any $x\in W^*(G)$.
\end{proof}

\begin{definition}
The \emph{tensor product representation} of two normal $*$-representation $(\pi_1,\mathcal{H}_1)$ and $(\pi_2,\mathcal{H}_2)$ of $W^*(G)$ is defined as $(\pi_1\fancyot \pi_2,\mathcal{H}_1\otimes\mathcal{H}_2)$, where $\pi_1\fancyot \pi_2:=(\pi_1\otimes \pi_2)\circ\hat\delta_G$. The \emph{conjugate representation} of $(\pi,\mathcal{H})$ is defined as $(\pi^c,\overline{\mathcal{H}})$, where $\overline{\mathcal{H}}$ is the conjugate Hilbert space of $\mathcal{H}$ and $\pi^c(x)\bar\xi:=\overline{\pi(\hat R_G(x^*))\xi}$ for any $x\in W^*(G)$ and $\bar\xi\in\overline{\mathcal{H}}$.
\end{definition}

We introduce the notion of tensor products of products quantized characters.

\begin{lemma}
For given two quantized characters $\chi_1,\chi_2$ the $\chi_1\fancyot\chi_2:=(\chi_1\otimes\chi_2)\circ\hat\delta_G$ becomes a quantized character.
\end{lemma}
\begin{proof}
It easy to see that $\chi_1\otimes\chi_2$ is a $\hat\tau^G\otimes\hat\tau^G$-KMS state on $W^*(G)\bar\otimes W^*(G)$. Thus, for any $x,y\in W^*(G)$ there exists a bounded continuous complex function $F$ on the zonal region $\{z\in\mathbb{C}\mid -1\leq\mathrm{Im}(z)\leq0\}$, which is analytic in $\{z\in\mathbb{C}\mid-1<\mathrm{Im}(z)<0\}$, such that
\[F(t)=\chi_1\otimes\chi_2(\delta(x)(\hat\tau^G\otimes\hat\tau^G)_t(\delta(y))),\]
\[F(t-\mathrm{i})=\chi_1\otimes\chi_2((\hat\tau^G\otimes\hat\tau^G)_t(\delta(y))\delta(x))\]
for any $t\in\mathbb{R}$. By Equation \eqref{eq:comm_action}, we have $F(t)=\chi_1\fancyot\chi_2(x\hat\tau^G_t(y))$ and $F(t-\mathrm{i})=\chi_1\fancyot\chi_2(\hat\tau^G_t(y)x)$ for any $t\in\mathbb{R}$. Therefore, $\chi_1\fancyot\chi_2$ is a $\hat\tau^G$-KMS state on $W^*(G)$.
\end{proof}

%%%%%%%%%%%%%%%%%%%%%%%%%%%%%%%%%%%%%%%%%%%%%%%%%%%%%%%%%%%%%%%%%%%%%%%%%%%%%%%%%%%%%%%%%%%%%%%%%%%%%%%%%%%%%%%%
\section{Quantum group $W^*$-algebras and their unitary representation theory}
\subsection{Quantum group $W^*$-algebras of inductive limits of compact quantum groups}
We will introduce the notion of inductive limits of compact quantum groups as inductive limits of quantum group $W^*$-algebras.

Let $\mathbb{G}=(G_N)_{N=0}^\infty$ be a sequence of compact quantum groups such that $G_0=(\mathbb{C},\mathrm{id}_\mathbb{C})$ and $G_N$ is a quantum subgroup of $G_{N+1}$, i.e., there is a surjective $*$-homomorphism $\theta_N\colon A(G_{N+1})\to A(G_N)$ satisfying $\delta_{G_N}\circ\theta_N=(\theta_N\otimes\theta_N)\circ\delta_{G_{N+1}}$, called the restriction map. Note that $\theta_N|_{\mathcal{A}(G_{N+1})}$ is also surjective (see \cite[Lemma 2.8]{Tomatsu}). By \cite[Lemma 2.10]{Tomatsu}, there exists a faithful normal unital $*$-homomorphism $\Theta_N\colon W^*(G_N)\to W^*(G_{N+1})$ satisfying $(\Theta_N\otimes\mathrm{id})(U^{G_N})=(\mathrm{id}\otimes \theta_N)(U^{G_{N+1}})$, where $\Theta_N$ is the dual map of $\theta_N$. By Equations \eqref{eq:coproduct}, we can prove that
\begin{equation}\label{eq:hrt}
\hat\delta_{G_{N+1}}\circ\Theta_N=(\Theta_N\otimes\Theta_N)\circ\hat\delta_{G_N}.
\end{equation}
Let $\{\hat\tau^{G_N}_t\}_{t\in\mathbb{R}}$ be the dual maps of $\{\tau^{G_N}_t\}_{t\in\mathbb{R}}$ and $\hat R_{G_N}$ the unitary antipode of $\mathcal{A}(G_N)$. Then, by \cite[Lemma 2.9]{Tomatsu}, we have
\begin{equation}\label{eq:comm_res}
\Theta_N\circ\hat\tau^{G_N}_t=\hat\tau^{G_{N+1}}_t\circ\Theta_N,\quad\Theta_N\circ\hat R_{G_N}=\hat R_{G_{N+1}}\circ\Theta_N.
\end{equation}

With those observations, by \cite[Theorem 7]{Takeda}, there exists a von Neumann algebra $W^*(G_\infty)$, where $G_\infty$ is just a symbol, and we have normal unital $*$-homomorphisms $\Theta_N^\infty\colon W^*(G_N)\to W^*(G_\infty)$ satisfying $\Theta_{N+1}^\infty\circ\Theta_N=\Theta_N^\infty$ and the following property: if a von Neumann algebra $M$ and normal unital $*$-homomorphisms $\Phi_N\colon W^*(G_N)\to M$ satisfying that $\Phi_{N+1}\circ\Theta_N=\Phi_N$ for every $N\geq0$, then there exists a unique normal $*$-homomorphism $\Phi\colon W^*(G_\infty)\to M$ satisfying $\Phi\circ\Theta_N^\infty=\Phi_N$ for every $N\geq0$. This $W^*(G_\infty)$ is the \emph{$W^*$-inductive limit} of the inductive system $(W^*(G_N),\Theta_N)_{N=0}^\infty$ in the sense of Takeda \cite{Takeda}. The \emph{Stratila--Voiculescu AF-algebra} $\mathfrak{A}(\mathbb{G})$ of the inductive system $\mathbb{G}$ can be defined as the $C^*$-subalgebra of $W^*(G_\infty)$ generated by $\bigcup_{N=0}^\infty \Theta_N^\infty(C^*(G_N))$.

\begin{remark}\label{Rem:inductivelimit}
Let $\mathfrak{M}(\mathbb{G}):=\varinjlim_N(W^*(G_N),\Theta_N)$ be the inductive limit in the category of $C^*$-algebras and $*$-homomorphisms. Remark that the both $C^*(G_N)\subset W^*(G_N)$ are faithfully embedded into $\mathfrak{M}(\mathbb{G})$ since $\Theta_N$ is injective for every $N\geq0$. A state $\varphi$ on $\mathfrak{M}(\mathbb{G})$ is said to be \emph{locally normal} (see e.g., \cite{Takesaki}) if $\varphi|_{W^*(G_N)}$ is normal for any $N\geq0$. Let $\mathcal{N}$ be the set of all locally normal states on $\mathfrak{M}(\mathbb{G})$. We define $(\pi,\mathcal{H}):=\bigoplus_{\varphi\in\mathcal{N}}(\pi_\varphi, \mathcal{H}_\varphi)$, where $(\pi_\varphi,\mathcal{H}_\varphi)$ is the GNS-representation of $\mathfrak{M}(\mathbb{G})$ associated with $\varphi\in\mathcal{N}$. Then a $W^*$-inductive limit of the inductive system $(W^*(G_N),\Theta_N)_{N=0}^\infty$ is obtained as $\pi(\mathfrak{M}(\mathbb{G}))''$. Indeed, the following facts are known:
\begin{enumerate}
\item $\pi|_{W^*(G_N)}$ is unital and normal for each $N\geq0$.
\item By \cite[Lemma 2]{Takeda}, $\mathcal{N}$ is dense in the set of states on $\mathfrak{M}(\mathbb{G})$ with respect to the weak${}^*$ topology. Thus, $\pi$ is an injective $*$-homomorphism.
\item By \cite[Theorem 1]{Takeda2} and \cite[Lemma 2]{Takeda}, $\mathcal{N}=\{\widetilde\varphi\circ \pi\mid\widetilde\varphi\text{ is normal state on }\pi(\mathfrak{M}(\mathbb{G}))''\}$.
\end{enumerate}
If $M$ is a von Neumann algebra and $\Phi_N\colon W^*(G_N)\to M$ are normal $*$-homomorphisms satisfying $\Phi_{N+1}\circ\Theta_N=\Phi_N$ for any $N\geq0$. By the universality of $\mathfrak{M}(\mathbb{G})$, we have a $*$-homomorphism $\Phi^0\colon\mathfrak{M}(\mathbb{G})\to M$ satisfying $\Phi^0|_{W^*(G_N)}=\Phi_N$ for any $N\geq0$. Then we have
\begin{align*}
\{\widetilde\varphi\circ\Phi^0\mid \widetilde\varphi\text{ is normal state on }M\}
\subseteq\mathcal{N}=\{\widetilde\varphi\circ \pi\mid\widetilde\varphi\text{ is normal state on }\pi(\mathfrak{M}(\mathbb{G}))''\}.
\end{align*}
Thus, by \cite[Theorem 2]{Takeda3}, there is a unique normal $*$-homomorphism $\Phi\colon \pi(\mathfrak{M}(\mathbb{G}))''\to M$ satisfying $\Phi\circ\Theta_N^\infty=\Phi_N$ for any $N\geq0$.
\end{remark}

Here we discuss a Hopf algebra structure of $W^*(G_\infty)$ with a distinguished flow. By \cite[Theorem 10]{Takeda}, we have
\[W^*(G_\infty)\bar\otimes W^*(G_\infty)=\overline{\bigcup_{N\geq0}\Theta_N^\infty(W^*(G_N))\bar\otimes\Theta_N^\infty(W^*(G_N))}^{\mathrm{ SOT}},\]
\[W^*(G_\infty)\bar\otimes W^*(G_\infty)\bar\otimes W^*(G_\infty)=\overline{\bigcup_{N\geq0}\Theta_N^\infty(W^*(G_N))\bar\otimes\Theta_N^\infty(W^*(G_N))\bar\otimes\Theta_N^\infty(W^*(G_N))}^\mathrm{SOT},\]
where the right-hand side of each of the above equations is the closure with respect to the strong operator topology. By Equation \eqref{eq:hrt}, we obtain the unital normal $*$-homomorphism $\hat\delta_\infty\colon W^*(G_\infty)\to W^*(G_\infty)\bar\otimes W^*(G_\infty)$ satisfying $\hat\delta_\infty\circ\Theta_N^\infty=(\Theta_N^\infty\otimes\Theta_N^\infty)\circ\hat\delta_N$ for any $N\geq0$. Then, by Equation \eqref{eq:coproduct}, we have 
$
(\mathrm{id}\otimes\hat\delta_\infty)\circ\hat\delta_\infty=(\hat\delta_\infty\otimes\mathrm{id})\circ\hat\delta_\infty.
$
By Equations \eqref{eq:comm_res}, we also obtain a one-parameter automorphism group $\{\hat\tau^\infty_t\}_{t\in\mathbb{R}}$ on $W^*(G_\infty)$ and an involutive normal $*$-anti-automorphism $\hat R_\infty\colon W^*(G_\infty)\to W^*(G_\infty)$ satisfying %and normal $*$-homomorphism $\hat\epsilon_N\colonW^*(G_\infty)\to\mathbb{C}$ satisfying 
$
\hat\tau^\infty_t\circ\Theta_N^\infty=\Theta_N^\infty\circ\hat\tau^{G_N}_t$ and $\hat R_\infty\circ\Theta_N^\infty=\hat R_{G_N}%,\quad \hat\epsilon_\infty\circ\Theta_N^\infty=\hat\epsilon_N
$
for every $N\geq0$. Remark that $\Theta_N^\infty\circ\hat R_{G_N}$ is normal $*$-homomorphism from $ W^*(G_N)$ to the opposite algebra of $W^*(G_\infty)$. Thus, we can apply the universality of $W^*(G_\infty)$ to the family of normal $*$-anti-homomorphisms $\Theta_N^\infty\circ\hat R_{G_N}$. 
%Then $\hat\epsilon_\infty$ becomes a counit of $W^*(G_\infty)$ and there exists a $\sigma$-weakly dense $*$-subalgebra $M$ of $W^*(G_\infty)$ such that for any $x\in M$
%\begin{equation}\label{eq:antipode}
%m\circ(\mathrm{id}\otimes(\hat R_\infty\circ\hat\tau^\infty_{-\mathrm{i}/2}))(\hat\delta_\infty(x))=\hat\epsilon_\infty(x)1=m\circ((\hat R_\infty\circ\hat\tau^\infty_{-\mathrm{i}/2})\otimes\mathrm{id})(\hat\delta_\infty(x)),
%\end{equation}
%where  $\hat\tau_{-\mathrm{i}/2}$ is the analytic continuation of $\{\hat\tau_t\}_{t\in\mathbb{R}}$ and $m\colonW^*(G_\infty)\otimesW^*(G_\infty)\to W^*(G_\infty)$ is the multiplication. 
Moreover, by Equations \eqref{eq:comm_action}, \eqref{eq:nogi}, for any $t\in\mathbb{R}$ we have
\begin{equation}\label{eq:prfm}
\hat\delta_\infty\circ\hat\tau^\infty_t=(\hat\tau^\infty_t\otimes\hat\tau^\infty_t)\circ\hat\delta_G,\quad\hat\tau^\infty_t\circ\hat R_\infty=\hat R_\infty\circ\hat\tau^\infty_t,\quad\hat\delta_\infty\circ\hat R_\infty=(\hat R_\infty\otimes\hat R_\infty)\circ\hat\delta_\infty^\mathrm{op},.
\end{equation}
where $\hat\delta_\infty^\mathrm{op}:=\hat\delta_\infty\circ\sigma$ and $\sigma\colon W^*(G_\infty)\bar\otimes W^*(G_\infty)\to W^*(G_\infty)\bar\otimes W^*(G_\infty)$ is the flip map.

Following Yamagami's work \cite{Yamagami} and Masuda and Nakagami's work \cite{MasudaNakagami}, we introduce a suitable notion of inductive limits of compact quantum groups and more generally a notion of not necessarily locally compact quantum group.
\begin{definition}\label{def:quantumgroup}
We call the above $G_\infty:=(W^*(G_\infty),\mathfrak{A}(\mathbb{G}),\hat\delta_\infty,\hat R_\infty,\{\hat\tau^\infty_t\}_{t\in\mathbb{R}})$ the \emph{inductive limit quantum group $W^*$-algebra} of $\mathbb{G}$. More generally, we call $G=(M,\mathfrak{A},\delta, R,\{\tau_t\}_{t\in\mathbb{R}})$ a \emph{quantum group $W^*$-algebra} if 
\begin{itemize}
\item $M$ is a von Neumann algebra,
\item $\mathfrak{A}$ is a $\sigma$-weakly dense $C^*$-subalgebra of $M$,
\item $\delta\colon M\to M\bar\otimes M$ is a comultiplication, i.e., unital normal $*$-homomorphism satisfying $(\mathrm{id}\otimes\delta)\circ\delta=(\delta\otimes\mathrm{id})\circ\delta$,
\item $R\colon M\to M$ is an involutive normal $*$-anti-homomorphism called the \emph{unitary antipode} and $\{\tau_t\}_{t\in\mathbb{R}}$ is a one-parameter automorphism group on $M$ preserving $\mathfrak{A}$ called the \emph{deformation automorphism group}. Moreover, $R$ and $\{\tau_t\}_{t\in\mathbb{R}}$ satisfy Equation \eqref{eq:prfm}.
\end{itemize}
\end{definition}

The necessity of $\mathfrak{A}$ might be unclear in the above definition. However, $\mathfrak{A}(\mathbb{G})$ is indeed necessary to give a natural topology on the convex set of quantized characters of $G_\infty$. See Theorem \ref{thm:char} and Remark \ref{rem:waidona} below. If we assume the existence of Haar weight $h$ of $G$, then the $(M, \delta, R,\{\tau_t\}_{t\in\mathbb{R}}, h)$ is a Woronowicz algebra in the sense of Masuda and Nakagami \cite{MasudaNakagami}, which is an approach to locally compact quantum groups based on von Neumann algebras. On the other hand, if $M$ is a direct sum of finite-dimensional von Neumann algebras and $G$ has a conuit $\epsilon$, then $(M,\delta,\epsilon, R,\{\tau_t\}_{t\in\mathbb{R}})$ is nothing less than a compact quantum group $W^*$-algebra. Here we do not assume anything about the algebraic structures of $M$. Instead, we impose $\mathfrak{A}$, which plays a role of the ``topology on $G$'', on our formulation of $G$.

%%%%%%%%%%%%%%%%%%%%%%%%%%%%%%%%%%%%%%%%%%%%%%%%%%%%%%%%%%%%%%%%%%%%%%%%%%%%%%%%%%%%%%%%%%%%%%%%%%%%%%%%%%%%%%%%%%%%
\subsection{Unitary representations and quantized characters of quantum group $W^*$-algebras}
Here we introduce the concepts of unitary representations and quantized characters of quantum group $W^*$-algebras. When a quantum group $W^*$-algebra comes from an inductive system of compact quantum groups, we compare the concepts of quantized characters in this paper and in our previous paper \cite{Sato1}. In the next two definitions and two lemmas, let $G=(M,\mathfrak{A},\delta, R,\{\tau_t\}_{t\in\mathbb{R}})$ be a general quantum group $W^*$-algebra.

\begin{definition}
A normal $\tau$-KMS state on $M$ with inverse temperature $-1$ is called a \emph{quantized character} of a quantum group $W^*$-algebra $G$. We denote by $\mathrm{Char}(G)$ the convex set of quantized characters of $G$ and equip $\mathrm{Char}(G)$ with the topology of pointwise convergence on $\mathfrak{A}$.
\end{definition}

\begin{lemma}
A quantized character $\chi\in\mathrm{Char}(G)$ is extreme if and only if $\chi$ is factorial.
\end{lemma}
\begin{proof}
By \cite[Theorem 2.3.19]{BratteliRobinson1}, we can show that extreme normal $\tau$-KMS states are also extreme points in the simplex of not necessarily normal $\tau$-KMS states. Thus, the claim immediately follows from \cite[Theorem 5.3.30]{BratteliRobinson2}.
\end{proof}

\begin{definition}
A normal $*$-representation $(\pi,\mathcal{H})$ of $M$ is called a \emph{unitary representation} of $G$ if there exists a one-parameter automorphism group $\{\gamma_t\}_{t\in\mathbb{R}}$ on $\pi(M)$ satisfying $\gamma_t(\pi(x))=\pi(\tau_t(x))$ for any $x\in M$ and $t\in\mathbb{R}$. A unitary representation $(\pi,\mathcal{H})$ of $G$ is called a \emph{factor} representation if $\pi(M)$ is a factor. A unitary representation $(\pi,\mathcal{H})$ of $G$ is called a \emph{quantum finite} (resp. \emph{semifinite}) representation if $\{\gamma_t\}_{t\in\mathbb{R}}$ is the modular automorphism group of some faithful normal state (resp. semifinite weight) on $\pi(M)$.
\end{definition}

It is easy to see that the GNS-representation of any quantized characters of $G$ is a quantum finite representation of $G$. 
%The following proposition gives a typical example of quantum finite representations of $\mathbb{G}$.
%\begin{proposition}
%The GNS-representation of quantized character of $\mathbb{G}$ is a quantum finite representation of $\mathbb{G}$. 
%\end{proposition}
%\begin{proof}
%Let $(T_\chi,\mathcal{H}_\chi,\xi_\chi)$ is the GNS-triple associated with $\chi\in\mathrm{Char}(\mathbb{G})$. Then the mapping $\gamma_t\colon T(x)\in T_\chi(W^*)\mapsto T(\hat\tau_t(x))\in T_\chi(W^*)$ is well defined. Indeed, if $T_\chi(x)=T_\chi(y)$ for $x,y\inW^*$ then 
%\begin{align*}
%\langle T(\hat\tau_t(x))T(a)\xi_\chi,T(b)\xi_\chi\rangle
%&=\langle T(x)T(\hat\tau_{-t}(a))\xi_\chi,T(\hat\tau_{-t}(b))\xi_\chi\rangle\\
%&=\langle T(y)T(\hat\tau_{-t}(a))\xi_\chi,T(\hat\tau_{-t}(b))\xi_\chi\rangle\\
%&=\langle T(\hat\tau_t(y))T(a)\xi_\chi,T(b)\xi_\chi\rangle
%\end{align*}
%for any $a,b\inW^*$ since $\chi$ is $\hat\tau$-invariant. Thus, $T(\hat\tau_t(x))=T(\hat\tau_t(y))$. Moreover, $\{\gamma_t\}_{t\in\mathbb{R}}$ is the modular automorphism group of the vector state with respect to $\xi_\chi$.
%\end{proof}

\begin{definition}
The \emph{tensor product} representation of two representation $(\pi_1,\mathcal{H}_1),\,(\pi_2,\mathcal{H}_2)$ of $G$ is defined as $(\pi_1\fancyot \pi_2,\mathcal{H}_1\otimes\mathcal{H}_2)$, where $\pi_1\fancyot \pi_2:=(\pi_1\otimes \pi_2)\circ\delta$. The \emph{contragredient} representation of $(\pi,\mathcal{H})$ is defined as $(\pi^c,\overline{\mathcal{H}})$, where $\pi^c(x)\bar\xi:=\overline{\pi( R(x^*))\xi}$ for any $x\in M$ and $\bar\xi\in\overline{\mathcal{H}}$.
\end{definition}

%It is easy to show that if $(T,\mathcal{H})$ is a representation of $\mathbb{G}$ then $(T^c,\overline{\mathcal{H}})$ is also a representation of $\mathbb{G}$. We assume that $(T,\mathcal{H})$ is quantum finite, i.e., there is $\gamma$-KMS state $\varphi$ (with the inverse temperature $\beta$) on $T(W^*)$, where $\gamma$ is the flow on $T(W^*)$ satisfying $\gamma_t(T(x))=T(\hat\tau_t(x))$ for any $x\inW^*$. Then $T^c(x)\in T^c(W^*)\mapsto\varphi(T(\hat R(x)))$ is well-defined and a $\bar\gamma$-KMS state (with the inverse temperature $-\beta$), where $\bar\gamma$ is the flow on $T^c(W^*)$ satisfying $\bar\gamma_t(T^c(x))=T^c(\hat\tau_t(x))$ for any $x\inW^*$.

It is easy to show that tensor product representations and contragredient representations of unitary representations are also unitary. However, it is unclear to us whether the contragredient representation of a quantum finite representation is again quantum finite. By the following lemma, the tensor product of any given two quantized characters is also a quantized character.
\begin{lemma}\label{L:Yss}
For two quantized characters $\chi_1,\chi_2$, the $\chi_1\fancyot\chi_2:=(\chi_1\otimes\chi_2)\circ\delta$ becomes a quantized character of $G$.
\end{lemma}
\begin{proof}
Remark that $\chi_1\otimes\chi_2$ is a $\tau\otimes\tau$-KMS state on $M\bar\otimes M$. Thus, for any pair $x,y\in M$ there exists a bounded continuous complex function $F$ on $\{z\in\mathbb{C}\mid -1\leq\mathrm{Im}(z)\leq0\}$, which is analytic in $\{z\in\mathbb{C}\mid-1<\mathrm{Im}(z)<0\}$, such that for any $t\in\mathbb{R}$
\[F(t)=\chi_1\otimes\chi_2(\delta(x)(\tau\otimes\tau)_t(\delta(y))),\]
\[F(t-\mathrm{i})=\chi_1\otimes\chi_2((\tau\otimes\tau)_t(\delta(y))\delta(x)).\]
By Equation \eqref{eq:prfm}, we have $F(t)=\chi_1\fancyot\chi_2(x\tau_t(y))$ and $F(t-\mathrm{i})=\chi_1\fancyot\chi_2(\tau_t(y)x)$ for any $t\in\mathbb{R}$. Hence, $\chi_1\fancyot\chi_2$ is $\tau$-KMS state.
\end{proof}

In the rest of this section, let $G_\infty:=(W^*(G_\infty),\mathfrak{A}(\mathbb{G}),\hat\delta_\infty,\hat R_\infty,\{\hat\tau^\infty_t\}_{t\in\mathbb{R}})$ be the inductive limit quantum group $W^*$-algebra of inductive system $\mathbb{G}=(G_N)_{N=0}^\infty$. We will compare the above definition of quantized character of $G_\infty$ with our previous one \cite[Definition 2.2]{Sato1}. Recall that $\{\hat\tau^\infty_t\}_{t\in\mathbb{R}}$ preserves $\mathfrak{A}(G_\infty)$ globally. We denote by the same symbol $\hat\tau^\infty_t$ its restriction to $\mathfrak{A}(G_\infty)$. Let $(e_{N,i})_{i=0}^\infty$ be an approximate unit for $C^*(G_N)$ for every $N\geq1$.

\begin{lemma}\label{L2.4}
Let $(\pi,\mathcal{H})$ be a non-degenerate $*$-representation of $\mathfrak{A}(\mathbb{G})$. Then there is a unique normal extension $\widetilde\pi\colon W^*(G_\infty)\to B(\mathcal{H})$ of $\pi$ if and only if $\pi(e_{N,i})$ converges to $\mathrm{id}_\mathcal{H}$ in the strong operator topology for every $N\geq1$.
\end{lemma}
\begin{proof}
If there is a normal extension $\widetilde\pi\colon W^*(G_\infty)\to B(\mathcal{H})$, then $\pi(e_{N,i})$ converges to $\pi(1)=\mathrm{id}_\mathcal{H}$ in the strong operator topology for every $N\geq1$. %See \cite[Theorem I\hspace{-.1em}I\hspace{-.1em}I.2.1.4]{Blackadar} for example. 
Conversely, we assume that $\pi(e_{N,i})$ converges to $\mathrm{id}_\mathcal{H}$ in the strong operator topology for every $N\geq1$. Then the restriction of $(\pi,\mathcal{H})$ to $C^*(G_N)$ is also non-degenerate for every $N\geq1$, and we can prove that there exists a unique extension of $(\pi|_{C^*(G_N)},\mathcal{H})$ to a normal $*$-homomorphism $(\widetilde \pi_N,\mathcal{H})$ of $W^*(G_N)$. By the uniqueness of normal extensions, we have $\widetilde \pi_{N+1}\circ\Theta_N=\widetilde \pi_N$. Therefore, by the universality of $W^*(G_\infty)$, we have a unique normal $*$-representation $(\widetilde \pi,\mathcal{H})$ satisfying that $\widetilde \pi|_{\mathfrak{A}(\mathbb{G})}=\pi$.
\end{proof}

\begin{lemma}\label{L:aki}
Let $\chi$ be a $\hat\tau^\infty$-KMS state on $\mathfrak{A}(\mathbb{G})$ and $(\pi_\chi,\mathcal{H}_\chi,\xi_\chi)$ the associated GNS-triple. Then $\|\chi|_{C^*(G_N)}\|=1$ for every $N\geq1$ if and only if the $*$-representation $(\pi_\chi,\mathcal{H}_\chi)$ satisfies the equivalent conditions in Lemma \ref{L2.4}.
\end{lemma}
\begin{proof}
Assume that $\|\chi|_{C^*(G_N)}\|=1$ for every $N\geq1$. Let $\mathfrak{A}(\mathbb{G})_{\hat\tau^\infty}$ be the set of $\hat\tau^\infty$-analytic elements. For any $x\in\mathfrak{A}(\mathbb{G})_{\hat\tau^\infty}$ we have 
\[\|\pi_\chi(x)\xi_\chi-\pi_\chi(e_{N,i})\pi_\chi(x)\xi_\chi\|^2=\chi(\hat\tau^\infty_{-\mathrm{i}}(x)x^*(1-e_{N,i})^2)\leq\|x\hat\tau^\infty_\mathrm{i}(x^*)\|\|\pi_\chi(1-e_{N,i})^2\xi_\chi\|.\]
By $\|\chi|_{C^*(G_N)}\|=1$, we have $\lim_{i\to\infty}\chi(e_{N,i})=\lim_{i\to\infty}\chi(e_{N,i}^2)=1$, see \cite[Proposition 2.3.11(2)]{BratteliRobinson1}. Namely, $\lim\sup_{i\to\infty}\|\pi_\chi(1-e_{N,i})^2\xi_\chi\|=0$ since $\|\pi_\chi(1-e_{N,i})^2\xi_\chi\|\leq2\|\pi_\chi(1-e_{N,i})\xi_\chi\|$. Therefore, $\pi_\chi(e_{N,i})$ converges to $\mathrm{id}_{\mathcal{H}_\chi}$ in the strong operator topology, since $\mathfrak{A}(\mathbb{G})_{\hat\tau^\infty}$ is norm dense in $\mathfrak{A}(\mathbb{G})$ and hence $\pi_\chi(\mathfrak{A}(\mathbb{G})_{\hat\tau^\infty})\xi_\chi$ is dense in $\mathcal{H}_\chi$. Conversely, if $\pi_\chi(e_{N,i})$ converges to $\mathrm{id}_{\mathcal{H}_\chi}$ in the strong operator topology for every $N\geq1$, then we have 
\[1\geq\|\chi|_{C^*(G_N)}\|\geq\limsup_{i\to\infty}\chi(e_{N,i})=\lim_{i\to\infty}\langle \pi_\chi(e_{N,i})\xi_\chi,\xi_\chi\rangle=1.\]
\end{proof}

Let $\mathrm{KMS}(\mathfrak{A}(\mathbb{G}))^0$ be the set of $\hat\tau^\infty$-KMS states $\varphi$ on $\mathfrak{A}(\mathbb{G})$ satisfying $\|\varphi|_{C^*(G_N)}\|=1$ for every $N\geq1$. We equip $\mathrm{KMS}(\mathfrak{A}(\mathbb{G}))^0$ with the weak${}^*$ topology induced from $\mathfrak{A}(\mathbb{G})$. By the following theorem, the above definition of quantized characters of $G_\infty$ and our previous one \cite[Definition 2.2]{Sato1} agree with each other.
\begin{theorem}\label{thm:char}
For any $\chi\in\mathrm{Char}(G_\infty)$ its restriction $\chi|_{\mathfrak{A}(\mathbb{G})}$ falls into $\mathrm{KMS}(\mathfrak{A}(\mathbb{G}))^0$. Moreover, this correspondence is affine and homeomorphic.
\end{theorem}
\begin{proof}
By Lemma \ref{L2.4}, \ref{L:aki}, $\chi|_{\mathfrak{A}(\mathbb{G})}$ falls into $\mathrm{KMS}(\mathfrak{A}(\mathbb{G}))^0$ for any $\chi\in\mathrm{Char}(G_\infty)$ and this correspondence is affine and bijective. Moreover, by the definition of topology on $\mathrm{Char}(G_\infty)$, this correspondence is also homeomorphic.
\end{proof}

\begin{corollary}
$\mathrm{Char}(G_\infty)$ is a Choquet simplex.
\end{corollary}
\begin{proof}
If follows from the above theorem and \cite[Proposition 2.4]{Sato1}.
\end{proof}

\begin{remark}\label{rem:waidona}
By Proposition \ref{prop:reason}, it is easy to show that $\mathrm{Char}(G_\infty)$ with the topology of pointwise convergence on $\mathfrak{A}(\mathbb{G})$ is homeomorphic to $\mathrm{Char}(G_\infty)$ with the topology of pointwise convergence on $\bigcup_{N\geq0}W^*(G_N)$. For a while, we assume that $G_\infty$ is not just a symbol, but an inductive limit of usual compact group, i.e., $G_\infty=\varinjlim_NG_N$ (equipped with the inductive limit topology). Then the topology on the set of characters of $G_\infty$ by the uniform convergence on compact subsets is equivalent to the weakest topology such that the mappings $\chi\mapsto\int_{G_N}\chi(g)f(g)d\mu_N(g)$ are continuous for any $f\in L^1(G_N,\mu_N)$ and any $N\geq1$, where $\mu_N$ is the Haar probability measure of $G_N$ (see \cite{Hirai}). Moreover, the set of characters of $G_\infty$ is homeomorphic to the set of tracial states $\varphi$ on $\mathfrak{A}(G_\infty)$ satisfying that $\|\varphi|_{C^*(G_N)}\|=1$ for every $N\geq1$ (see \cite[Section 2.1]{EnomotoIzumi} and \cite[Section 2.4]{Sato1}). By Theorem \ref{thm:char}, they are also homeomorphic to the set of normal tracial states on $W^*(G_\infty)$ with the topology of pointwise convergence on $\mathfrak{A}(G_\infty)$, which is nothing less than the set of quantized characters since $\{\hat\tau^\infty_t\}_{t\in\mathbb{R}}$ is trivial in this case. 
\end{remark}

\begin{remark}
We can regard $W^*(G_\infty)$ as the usual group von Neumann algebra of $G_\infty$. Indeed, if $G_\infty$ is an inductive limit of usual compact groups, then we have a bijective correspondence between the set of all unitary representations of $G_\infty$ and the set of all non-degenerate normal $*$-representations of $W^*(G_\infty)$. Let $(\pi,\mathcal{H})$ be a unitary representation of $G_\infty$. Then each restriction $\pi|_{G_N}$ induces a unique non-degenerate representation $(\tilde\pi_N,\mathcal{H})$ of $W^*(G_N)$, and we have $\tilde\pi_{N+1}\circ\Theta_N=\tilde\pi_N$ for every $N$. Thus, we obtain a non-degenerate representation $(\tilde\pi,\mathcal{H})$ of $W^*(G_\infty)$ satisfying $\tilde\pi\circ\Theta_N^\infty=\tilde\pi_N$ for every $N$. We can also obtain the reverse correspondence in the same way. Moreover, we have $\tilde\pi(W^*(G_\infty))''=\pi(G_\infty)''$. Thus, $W^*(G_\infty)$ and $G_\infty$ must have the same representation theory.
\end{remark}

%%%%%%%%%%%%%%%%%%%%%%%%%%%%%%%%%%%%%%%%%%%%%%%%%%%%%%%%%%%%%%%%%%%%%%%%%%%%%%%%%%%%%%%%%%%%%%%%%%%%%%%%%%%%%%%%
\section{The inductive limit $U_q(\infty)$ and their tensor product representations}
\subsection{Basics of quantum unitary groups $U_q(N)$}
Let $q\in(0,1)$ and $U_q(N)$ the quantum unitary group of rank $N$. See e.g., \cite{NoumiYamadaMimachi}, \cite{KilSch} for its definition. We can regard $U_q(N)$ as a quantum subgroup $U_q(N+1)$. Thus, their inductive system $\mathbb{U}_q$ is well defined. It is known that the branching rule of irreducible representations of $U_q(N)$ does not depend on the quantization parameter $q$, that is, the following facts hold (see \cite{Zelobenko}, \cite{NoumiYamadaMimachi}):
\begin{itemize}
\item All the equivalence classes of irreducible representations of $U_q(N)$ and $U(N)$ are labeled with the set of \emph{signatures} given as $\mathrm{Sign}_N:=\{\lambda=(\lambda_n)_{n=1}^N\in\mathbb{Z}^N\mid \lambda_1\geq\lambda_2\geq\cdots\geq\lambda_N\}$, and we set $\mathrm{Sign}_0:=\{*\}$.
\item The restriction of the irreducible representation labeled with $\nu\in\mathrm{Sign}_{N+1}$ contains the irreducible one labeled with $\lambda\in\mathrm{Sign}_N$ if and only if $\nu_1\geq\lambda_1\geq\nu_2\geq\cdots\geq\lambda_N\geq\nu_{N+1}$. We write $\lambda\prec\nu$ in this case. Moreover, we assume that $*\prec\lambda$ for every $\lambda\in\mathrm{Sign}_1$.
\end{itemize} 

\begin{remark}
By the above two facts, $U_q(N)$ and $U(N)$ have the same fusion ring and the dimension of each irreducible representation does not depend on $q$. Therefore, by \cite[Theorem 2.7.10, Proposition 2.7.12]{NeshveyevTuset}, the Haar state of $U_q(N)$ is faithful.
% characters = schur polynomials
\end{remark}

\begin{remark}\label{Rem:3.2}
Noumi, Yamada and Mimachi (\cite[Theorem 2.5]{NoumiYamadaMimachi}) gave a concrete realization of irreducible representations of $U_q(N)$. In particular, the irreducible representation of $U_q(N)$ corresponding to $(k,\dots,k)$ for any $k\in\mathbb{Z}$ is given as $\mathrm{det}_q^{k}(N)\in\mathbb{C}\otimes A(U_q(N))$, where $\mathrm{det}_q(N)$ is the quantum determinant. See \cite{NoumiYamadaMimachi} for more details.
\end{remark}

Let $T(N)=(C(\mathbb{T}^N),\delta_{T(N)})$ be the compact quantum group of $N$ dimensional torus $\mathbb{T}^N$ (see \cite[Example 1.2.2]{NeshveyevTuset}). Then we can regard $T(N)$ is a quantum subgroup of $U_q(N)$. See \cite{KilSch} for example. We denote by $\pi_N$ the restriction map.

%%%%%%%%%%%%%%%%%%%%%%%%%%%%%%%%
\subsection{$q$-Coherent systems and $q$-Schur generating functions}
Here we briefly summarize $q$-coherent systems and $q$-Schur generating functions. We define $w_q(\lambda,\nu):=q^{(N+1)|\lambda|-N|\nu|}$ for any pair $\lambda\in\mathrm{Sign}_N$, $\nu\in\mathrm{Sign}_{N+1}$ with $\lambda\prec\nu$. See \cite{Sato1} for the representation-theoretic meaning of this definition. A sequence of probability measures $P_N$ on $\mathrm{Sign}_N$ is called a \emph{$q$-coherent system} if 
\[\frac{P_N(\{\lambda\})}{\dim_q(\lambda)}=\sum_{\nu\in\mathrm{Sign}_{N+1}:\lambda\prec\nu}w_q(\lambda,\nu)\frac{P_{N+1}(\{\nu\})}{\dim_q(\nu)}\]
holds for every $\lambda\in\mathrm{Sign}_N$ and $N\geq1$, where $\dim_q(\lambda)$ is the quantum dimension of $\lambda\in\mathrm{Sign}_N=\widehat{U_q(N)}$. 

For a probability measure $P_N$ on $\mathrm{Sign}_N$ its \emph{$q$-Schur generating function} $\mathcal{S}(z_1,\dots,z_N;P_N)$ is defined as
\[\mathcal{S}(x_1,\dots,x_N;P_N):=\sum_{\lambda\in\mathrm{Sign}_N}P_N(\{\lambda\})\frac{s_\lambda(x_1,\dots,x_N)}{s_\lambda(1,q^{-2},\dots,q^{-2(N-1)})},\]
where $s_\lambda(x_1,\dots,x_N)$ is the Schur (Laurent) polynomial with label $\lambda\in\mathrm{Sign}_N$. Remark that $S(x_1,\dots,x_N;P_N)$ absolutely converges on $\mathcal{T}_N:=\{(x_1,\dots,x_N)\in\mathbb{C}^N\mid |x_i|=q^{-2(i-1)}\}$ for any probability measure $P_N$ on $\mathrm{Sign}_N$. See \cite[Proposition 4.5]{Gorin12}. 

In what follows, we use the same symbol $U_q(N)$ to denote the associated compact quantum group $W^*$-algebra. For any quantized character $\chi$ of $U_q(N)$ there is a unique probability measure $P_N$ on $\mathrm{Sign}_N$ satisfying that
\begin{equation}\label{eq:char_prob}
(\chi\otimes\pi_N)(U^{U_q(N)})(z_1,\dots,z_N)=\mathcal{S}(z_1,q^{-2}z_2,\dots,q^{-2(N-1)}z_N;P_N)
\end{equation}
on $(z_1,\dots,z_N)\in\mathbb{T}^N$. See Section \ref{sec:cqg} for the notation $U^{U_q(N)}$. In particular, we have
\begin{equation}\label{eq:char_irr}
(\chi^\lambda\otimes\pi_N)(U^{U_q(N)})(z_1,\dots,z_N)=\frac{s_\lambda(z_1,q^{-2}z_2,\dots,q^{-2(N-1)}z_N)}{s_\lambda(1,q^{-2},\dots,q^{-2(N-1)})},
\end{equation}
where $\chi^\lambda$ is the indecomposable quantized character with label $\lambda\in\mathrm{Sign}_N$. 

We denote by $U_q(\infty)$ the inductive limit quantum group $W^*$-algebra of $\mathbb{U}_q$, called the \emph{infinite-dimensional quantum unitary group}. By Theorem \ref{thm:char} and \cite[Section 3]{Sato1}, two simplexes of $q$-coherent systems and quantized characters of $U_q(\infty)$ are affine homeomorphic. Indeed, Equation \eqref{eq:char_prob} gives an affine homeomorphism between them. Moreover, their extreme points are completely parametrized by $\mathcal{N}:=\{(\theta_i)_{i=1}^\infty\in\mathbb{Z}^\infty\mid\theta_1\leq\theta_2\leq\cdots\}$. See \cite{Sato1}, \cite{Gorin12} for more details. In particular, if $\chi$ is the quantized character of $U_q(\infty)$ corresponding to a $q$-coherent system $(P_N)_{N=1}^\infty$, then we have
\begin{equation}\label{eq:char_prob_inf}
(\chi|_{W^*(U_q(N))}\otimes\pi_N)(U^{U_q(N)})(z_1,\dots,z_N)=\mathcal{S}(z_1,q^{-2}z_2,\dots,q^{-2(N-1)}z_N;P_N).
\end{equation}

%Moreover, if $\chi$ is extreme and $(\theta_i)_{i=1}^\infty\in\mathcal{N}$ is its corresponding parameter, then we have
%\begin{equation}\label{eq:char_irr_inf}
%(\chi|_{W^*(U_q(N))}\otimes\pi_N)(U^{U_q(N)})(z_1,\dots,z_N)=\lim_{L\to\infty}\frac{s_{(\theta_L,\theta_{L-1},\dots,\theta_1)}(z_1,q^{-2}z_2,\dots,q^{-2(N-1)}z_N,q^{-2N},\dots,q^{-2(L-1)})}{s_{(\theta_L,\theta_{L-1},\dots,\theta_1)}(1,q^{-2},\dots,q^{-2(N-1)},q^{-2N},\dots,q^{-2(L-1)})}.
%\end{equation}

%Here we remark that representation theoretic (and ergodic theoretic) aspects of $\mathcal{N}$ are investigated in \cite{Sato2}. Namely, since GNS-representations associated with extreme KMS states are factorial, we obtain von Neumann factors corresponding to $\mathcal{N}$. Then Murray--von Neumann--Connes types of resulting von Neumann factors are completely determined in terms of $\mathcal{N}$.

%%%%%%%%%%%%%%%%%%%%%%%%%%%%%%%%%%%%%%%%%%%%%%%%%%%%%%%%%%%%%%%%%%%%%%%%%%%%%%%%%%%%%%%%%%%%%%%%%%%%%%%%%%%%%%%%
\subsection{Tensor product representation of $U_q(\infty)$ and $q$-Schur generating functions}
In this section, we discuss tensor product representations of $U_q(\infty)$ and give a representation theoretic interpretation of the transformations $A_k$ on $\mathcal{N}$ defined as $A_k((\theta_i)_{i=1}^\infty):=(\theta_i+k)_{i=1}^\infty$, which often appear in the analysis of $q$-central probability measures (see \cite{Gorin12}, \cite{Cuenca}).

\begin{proposition}\label{Prop:suzu}
Let $\chi,\chi_1,\chi_2$ be quantized characters of $U_q(N)$ and $P,P_1,P_2$ their corresponding probability measures on $\mathrm{Sign}_N$, respectively. Then $\chi=\chi_1\fancyot \chi_2$ if and only if 
\begin{equation}\label{eq:prop3.1}
\mathcal{S}(x_1,\dots,x_N;P)=\mathcal{S}(x_1,\dots,x_N;P_1)\mathcal{S}(x_1,\dots,x_N;P_2)
\end{equation}
holds for every $(x_1,\dots,x_N)\in\mathcal{T}_N$.
\end{proposition}
\begin{proof}
Let $U_N:=U^{U_q(N)}$. By Equations \eqref{eq:coproduct}, we have
\begin{align*}
[(\chi_1\fancyot\chi_2)\otimes\pi_N](U_N)
&=(\chi_1\otimes\chi_2\otimes\pi_N)(U_{N23}U_{N13})\\
&=(\chi_1\otimes\chi_2\otimes\mathrm{id})((\mathrm{id}\otimes\pi_N)(U_N)_{23}(\mathrm{id}\otimes\pi_N)(U_N)_{13})\\
%&=(\chi_1\otimes\mathrm{id}\otimes\mathrm{id})([(\mathrm{id}\otimes\pi_N)(U_N)_{13}\cdot(\mathrm{id}\otimes\chi_2\otimes\mathrm{id})]((\mathrm{id}\otimes\pi_N)(U_N)_{23}))\\
%&=(\chi_1\otimes\mathrm{id}\otimes\mathrm{id})([(\mathrm{id}\otimes\pi_N)(U_N)_{13}\cdot\mathrm{id}^{\otimes3}]\circ(\mathrm{id}\otimes\chi_2\otimes\mathrm{id})((\mathrm{id}\otimes\pi_N)(U_N)_{23}))\\
%&=(\chi_1\otimes\mathrm{id}\otimes\mathrm{id})([(\mathrm{id}\otimes\pi_N)(U_N)_{13}\cdot\mathrm{id}^{\otimes3}]((\chi_2\otimes\pi_N)(U_N)_{23}))\\
%&=(\chi_1\otimes\mathrm{id}\otimes\mathrm{id})((\chi_2\otimes\pi_N)(U_N)_{23}(\mathrm{id}\otimes\pi_N)(U_N)_{13})\\
%&=[(\chi_1\otimes\mathrm{id}\otimes\mathrm{id})\cdot(\chi_2\otimes\pi_N)(U_N)_{23}]((\mathrm{id}\otimes\pi_N)(U_N)_{13})\\
%&=[\mathrm{id}^{\otimes3}\cdot(\chi_2\otimes\pi_N)(U_N)_{23}]\circ(\chi_1\otimes\mathrm{id}\otimes\mathrm{id})((\mathrm{id}\otimes\pi_N)(U_N)_{13})\\
%&=[\mathrm{id}^{\otimes3}\cdot(\chi_2\otimes\pi_N)(U_N)_{23}](\chi_1\otimes\otimes\pi_N)(U_N)_{13})\\
&=(\chi_1\otimes\pi_N)(U_N)_{13}(\chi_2\otimes\pi_N)(U_N)_{23}
\end{align*}
Thus, by Equation \eqref{eq:char_prob}, $\chi=\chi_1\fancyot\chi_2$ if and only if Equation \eqref{eq:prop3.1} holds.
\end{proof}

This type of claim holds even in the infinite-dimensional case.

\begin{theorem}\label{prop:3.2}
Let $\chi,\chi_1,\chi_2$ be quantized characters of $U_q(\infty)$ and $(P_N)_N,(P_{1,N})_N,(P_{2,N})_N$ their corresponding $q$-coherent systems, respectively. Then $\chi=\chi_1\fancyot\chi_2$ if and only if Equation \eqref{eq:prop3.1} holds for every $N\geq1$.
\end{theorem}
\begin{proof}
Since $\bigcup_{N\geq0}W^*(U_q(N))$ is $\sigma$-weakly dense in $W^*(U_q(\infty))$, we obtain that $\chi=\chi_1\fancyot\chi_2$ if and only if $\chi|_{W^*(U_q(N))}=(\chi_1\fancyot\chi_2)|_{W^*(U_q(N))}$ for every $N\geq1$. Thus, this theorem immediately follows from Proposition \ref{Prop:suzu}.
\end{proof}

\begin{corollary}\label{theorem:hiro}
Let $\chi_\theta$ be the extreme quantized character corresponding to $\theta\in\mathcal{N}$. Then $\chi_\theta\fancyot\chi_{(k,k,\dots)}=\chi_{A_k(\theta)}$ for every $k\in\mathbb{Z}$.
\end{corollary}
\begin{proof}
By Proposition \ref{prop:3.2}, it suffices to show that for every $N\geq1$
\[\mathcal{S}(x_1,\dots,x_N;P^{A_k(\theta)}_N)=\mathcal{S}(x_1,\dots,x_N;P^\theta_N)\mathcal{S}(x_1,\dots,x_N;P^{(k,k,\dots)}_N),\]
where $(P^\theta_N)_{N=1}^\infty$ is the extreme $q$-coherent system corresponding to $\theta\in\mathcal{N}$. We define $A_k(\lambda):=(\lambda_1+k,\dots,\lambda_N+k)$ for any $\lambda\in\mathrm{Sign}_N$. By \cite[Proposition 5.13]{Gorin12}, we have $P_N^{A_k(\theta)}=P_N^\theta\circ A_{-k}$. Thus, we have 
\[\mathcal{S}(x_1,\dots,x_N;P^{A_k(\theta)}_N)=x_1^k(q^2x_2)^k\cdots (q^{2(N-1)}x_N)^k\mathcal{S}(x_1,\dots,x_N;P^\theta_N)\] 
since $s_{A_k(\lambda)}(x_1,\dots,x_N)=x_1^k\cdots x_N^ks_\lambda(x_1,\dots,x_N)$. By \cite[Theorem 5.1, Proposition 4.10]{Gorin12},
\begin{align*}
\mathcal{S}(x_1,\dots,x_N;P^{(k,k,\dots)}_N)
&=\lim_{L\to\infty,N\leq L}\frac{s_{(k,\dots,k)}(x_1,\dots,x_N,q^{-2N},\dots,q^{-2(L-1)})}{s_{(k,\dots,k)}(1,q^{-2}\dots,q^{-2(L-1)})}\\
%&=(\chi^{(k,\dots,k)}\otimes\pi_N)(U^{U_q(N)})(x_1,q^2x_2,\dots,q^{2(N-1)}x_N)\\
%&=\frac{s_{(k,\dots,k)}(x_1,x_2,\dots,x_N)}{s_{(k,\dots,k)}(1,q^{-2},\dots,q^{-2(N-1)})}\\
&=x_1^k(q^2x_2)^k\dots(q^{2(N-1)}x_N)^k.
\end{align*}
Therefore, we obtain $\mathcal{S}(x_1,\dots,x_N;P^{A_k(\theta)}_N)=\mathcal{S}(x_1,\dots,x_N;P^\theta_N)\mathcal{S}(x_1,\dots,x_N;P^{(k)}_N)$.
\end{proof}

\begin{remark}
Recall that the set of extreme quantized characters of $U_q(\infty)$ is completely parametrized by $\mathcal{N}=\{(\theta_i)_{i=1}^\infty\in\mathbb{Z}^\infty\mid\theta_1\leq\theta_2\leq\cdots\}$. Since the GNS-representation associated with any extreme KMS state is factorial, we obtain a factor (a von Neumann algebra with trivial center) corresponding to any parameter of $\mathcal{N}$. In our previous paper \cite{Sato2}, the Murray--von Neumann--Connes type of the resulting factor is explicitly determined in terms of $\mathcal{N}$. Combining with the above corollary, we obtain the following rule on tensor products of factor representations associated with extreme quantized characters of $U_q(\infty)$:
\[\text{type X}\otimes\text{type I}_1=\text{type X},\]
where $X=$I${}_1$, I${}_\infty$ or I\hspace{-.1em}I\hspace{-.1em}I${}_{q^2}$. 
\end{remark}

The description of tensor product representations is of importance in representation theory. The determinant $\det(U)$ of $U\in U(N)\subset U(\infty)$ gives one of the simplest extreme characters, and the product of the determinant and an arbitrary extreme character also becomes an extreme character. Moreover, using Voiculescu functions (see \cite{Voiculescu76}) of extreme characters, we can explicitly describe such a multiplication of the determinant and an extreme character (i.e., the tensor product of the determinant and the finite factor representation associated with each extreme character). In the asymptotic representation theory of $U_q(\infty)$, Voiculescu functions are replaced with $q$-Schur generating functions. In fact, we have described the tensor product representations associated with any pair of quantized characters in terms of $q$-Schur generating functions in Theorem \ref{prop:3.2}. By Remark \ref{Rem:3.2}, the quantum analog of $\det{}^k$ ($k\in\mathbb{Z}$) is given as the extreme quantized character corresponding to $(k,k,\dots)\in\mathcal{N}$. Similarly to the case of $U(\infty)$, we have proved that the tensor product of extreme quantized character corresponding to $(k,k,\dots)\in\mathcal{N}$ and an arbitrary parameter must be extreme. Moreover, the resulting new parameter is given by the transformation $A_k$ on $\mathcal{N}$. Therefore, we have been able to give an explicit representation theoretic interpretation to $A_k$.

\appendix
%%%%%%%%%%%%%%%%%%%%%%%%%%%%%%%%%%%%%%%%%%%%%%%%%%%%%%%%%%%%%%%%%%%%%%%%%%%%%%%%%%%%%%%%%%%%%%%%%%%%%%%%%%%%%%%%
\section{Spherical representations and spherical functions}
This appendix is based on what Yoshimichi Ueda explained to us in a rather general framework. The goal of this appendix is to develop a minimal foundation of Olshanski's spherical representation theory in the quantum setting in order to provide a basis for subsequent investigations of unitary representation theory for $U_q(\infty)$. It may be regarded as an answer to a question to the first version of this paper asked by Grigori Olshanski. In what follows, we will freely use standard facts in modular theory, see e.g., \cite[Section 2.5]{BratteliRobinson1}, \cite[Chapter IV--IX]{Takesaki:book}. Actually, the theory of standard forms based on modular theory lies behind the materials here. We also use the standard notation $(a\cdot\varphi\cdot b)(x)=\varphi(bxa)$ with a state $\varphi$ on a $C^*$-algebra $A$ and $a,b,x\in A$.

Let $G$ be a topological group. A triple $(T,\mathcal{H},\xi)$ is called a \emph{spherical representation} if $(T,\mathcal{H})$ is a unitary representation of $G\times G$ and $\xi\in\mathcal{H}$ is a cyclic $G$-invariant unit vector, where we remark that $G$ can be identified a subgroup of $G\times G$ by $g\in G\mapsto (g,g)\in G\times G$. Then the function $\varphi\colon G\times G\to\mathbb{C}$ given as $\varphi(g,h):=\langle T(g,h)\xi,\xi\rangle$ is called a \emph{spherical function}. Namely, the spherical function $\varphi$ is a positive-definite continuous function satisfying that $\varphi(e)=1$ and $\varphi$ is $G$-biinvariant, that is, $\varphi(kgk',khk')=\varphi(g,h)$ for any $g,h,k,k'\in G$. It is known that there exists an affine bijective correspondence between the characters of $G$ and the spherical functions. In this way, there also exists a correspondence between finite factor representations of $G$ and irreducible spherical representations. See \cite{Olshanski90} for more details. %Moreover, the spherical representation theory is one of the starting points for the harmonic analysis on $U(\infty)$ due to Olshanski \cite{Olshanski90} and Borodin--Olshanski  \cite{BorodinOlshanski}. We believe that this appendix is a very first step toward developing a quantum analog of those works on harmonic analysis on $U(\infty)$, e.g., to reveal a representation-theoretic interpretation of the work due to Gorin and Olshanski \cite{GorinOlshanski}.

%In the viewpoint of operator algebras (or Connes's view), unitary representation theory can be understood in terms of bimodules of operator algebras. Following this viewpoint, the above correspondence between (extreme) characters and (irreducible) spherical functions can be interpreted as the purification map due to Woronowicz (see \cite{Woronowicz72}). Actually, following the idea of the purification map, we will define spherical functions as certain states on a suitable tensor product $M\otimes_\mathrm{bin} M$, the binormal tensor product. See \cite[Section IV.2.3]{Blackadar}.

Let $G=(M,\mathfrak{A},\delta, R,\{\tau_t\}_{t\in\mathbb{R}})$ be a quantum group $W^*$-algebra (see Definition \ref{def:quantumgroup}). We remark that the presence of dense $C^*$-algebra $\mathfrak{A}$ and comultiplication $\delta$ is not necessary in what follows. We denote by $M\otimes_\mathrm{bin} M$ the binormal tensor product of $M$, see \cite[Section IV.2.3]{Blackadar}. For any map $f$ on $M\otimes_\mathrm{bin} M$ we define $f_L$ and $f_R$ on $M$ by $f_L(x):=f(x\otimes 1)$, $f_R(x):=f(1\otimes x)$. If $f_L$ and $f_R$ are normal, then $f$ is said to be binormal. Note that the binormal tensor product $M\otimes_\mathrm{bin}M$ enjoys the universality with respect to two commuting \emph{normal} $*$-representations of $M$. We denote by $M_{\tau}$ the $\sigma$-weakly dense $*$-subalgebra of $\tau$-analytic elements.

\begin{definition}
A triple $(T,\mathcal{H},\xi)$ is called a \emph{spherical representation} if $(T,\mathcal{H})$ is a binormal $*$-representation of $M\otimes_\mathrm{bin} M$ and $\xi\in\mathcal{H}$ is a cyclic unit vector satisfying that 
\begin{equation}\label{eq:inv}
T(1\otimes x)\xi=T(\kappa(x)\otimes 1)\xi
\end{equation} for any $x\in M_{\tau}$, where $\kappa:=R\circ\tau_{-\mathrm{i}/2}=\tau_{-\mathrm{i}/2}\circ R$. Then $\xi$ is called a \emph{spherical vector}. Two spherical representations $(T_i,\mathcal{H}_i,\xi_i)$ ($i=1,2$) are unitarily equivalent if there exists a unitary intertwiner $U$ from $(T_1,\mathcal{H}_1)$ to $(T_2,\mathcal{H}_2)$ satisfying that $U\xi_1=\xi_2$.
\end{definition}

Recall that $\kappa$ becomes the canonical antipode when $G$ comes from a usual quantum group like $SU_q(2)$, etc, see \cite[Remark 1.3(a)]{MasudaNakagami}. In this viewpoint, Equation \eqref{eq:inv} can be regared as an analog of $T(e,g)\xi=T(g^{-1},e)\xi$ for every $g\in G$ with a unitary representation $(T,\mathcal{H})$ of an ordinary group $G\times G$ and $\xi\in\mathcal{H}$. This is clearly equivalent to that $T(g,g)\xi=\xi$ for every $g\in G$. Therefore, our definition is a natural generalization of Olshanski's one for ordinary groups.

Let $\chi$ be a quantized character of $G$ and $(\pi_\chi,\mathcal{H}_\chi,\xi_\chi)$ the associated GNS-triple. Since $\chi$ is a $\tau$-KMS state, $\xi_\chi$ is separating and cyclic for $\pi_\chi(M)$. This guarantees that $\chi$ ``extends'' to a faithful normal state $\hat \chi$ on $\pi_\chi(M)$, i.e., $\hat\chi(\pi_\chi(x))=\chi(x)$ for any $x\in M$. See \cite[Corollary 5.3.8]{BratteliRobinson2} and around there. Moreover, the modular automorphism group $\{\sigma^{\hat\chi}_t\}_{t\in\mathbb{R}}$ associated with $\hat\chi$ satisfies that $\sigma^{\hat\chi}_t(\pi_\chi(x))=\pi_\chi(\tau_t(x))$ for every $x\in M$ and $t\in \mathbb{R}$. Hence, the modular conjugation $J_{\hat\chi}\colon\mathcal{H}_\chi\to\mathcal{H}_{\hat\chi}$ is known to be given by the formula: $J_{\hat\chi}\pi_\chi(x)\xi_\chi=\pi_\chi(\tau_{-\mathrm{i}/2}(x^*))\xi_\chi$ for every $x\in M_\tau$. Then it is well known that $J_{\hat\chi}\pi_\chi(M)J_{\hat\chi}$ coincides with the commutant $\pi_\chi(M)'$. By universality of binormal tensor product, there is a unique binormal $*$-representation $(T_\chi,\mathcal{H}_\chi)$ of $M\otimes_\mathrm{bin} M$ such that $T_\chi(x\otimes y)=\pi_\chi(x)J_{\hat\chi}(\pi_\chi(R(y)^*))J_{\hat\chi}$ for any simple tensor $x\otimes y$.

\begin{lemma}\label{lem:CW}
The following three assertions hold true:
\begin{enumerate}
\item $(T_\chi,\mathcal{H}_\chi,\xi_\chi)$ is a spherical representation.
\item $(T_\chi,\mathcal{H}_\chi)$ is irreducible if and only if $\chi$ is extreme, that is, $(\pi_\chi,\mathcal{H}_\chi)$ is factorial.
\item $(T_\chi,\mathcal{H}_\chi,\xi_\chi)$ is a unique spherical representation up to unitarily equivalence satisfying that $\langle (T_\chi)_L(x)\xi_\chi,\xi_\chi\rangle=\chi(x)$ for any $x\in M$.
\end{enumerate}
\end{lemma}
\begin{proof}
The first and the second assertions are easy to prove. We leave them to the reader and prove only the third assertion. Let $(T,\mathcal{H},\xi)$ be a spherical representation satisfying $\langle T_L(x)\xi,\xi\rangle=\chi(x)$ for any $x\in M$. By the uniqueness of GNS-triple associated with $\chi$, it suffices to show that $\xi$ is cyclic for $T_L(M)$. By Equation \eqref{eq:inv}, we have $T(x\otimes y)\xi=T_L(x\kappa(y))\xi$ for any $x\in M$ and $y\in M_\tau$. Thus, $T(M\otimes_\mathrm{bin} M)\xi\subset \overline{T_L(M)\xi}^{\|\,\cdot\,\|}$ since $M_\tau$ is $\sigma$-weakly dense in $M$ and $T$ is binormal. Namely, $\xi$ is also cyclic for $T_L(M)$.
%Remark that $T_\chi$ is a well-defined $*$-homomorphism on $M\otimes_\mathrm{bin}  M$ by the universality of maximal tensor products. See e.g., \cite[Corollary I\hspace{-.1em}I.9.7.3]{Blackadar}. It is clear that $\xi_\chi$ is cyclic for $T_\chi(M\otimes_{\mathrm{max}}M)$ since $\xi_\chi$ is cyclic for $\pi_\chi(M)$. Moreover, $T_{\chi L}$ and $T_{\chi R}$ are normal since $\pi_\chi$ and $R$ are normal. For any $x\in M_{\tau}$ we have
%\[T_\chi(1\otimes x)\xi_\chi=J_\chi\pi_\chi( R(x)^*)J_\chi\xi_\chi=\pi_\chi(\kappa(x))\xi_\chi=T_\chi(\kappa(x)\otimes1)\xi_\chi.\]
%Therefore $(T_\chi,\mathcal{H}_\chi,\xi_\chi)$ is a spherical representation.
%
%Since $J_\chi\pi_\chi(M)J_\chi=\pi_\chi(M)'$ (see e.g., \cite[Theorem 2.5.14]{BratteliRobinson1}), we have $\pi_\chi(M)\cap\pi_\chi(M)'=T_\chi(M\otimes_\mathrm{bin}  M)'$. Therefore, $(T_\chi,\mathcal{H}_\chi)$ is irreducible if and only if $(\pi_\chi,\mathcal{H}_\chi)$ is factorial.
\end{proof}

For a spherical representation $(T,\mathcal{H},\xi)$ we define a state $\varphi:=\varphi_{(T,\mathcal{H},\xi)}$ on $M\otimes_\mathrm{bin}  M$ by $\varphi(X):=\langle T(X)\xi,\xi\rangle$ for any $X\in M\otimes_\mathrm{bin} M$, which we call the \emph{spherical function} associated with $(T,\mathcal{H},\xi)$. Recall that original spherical functions are defined as certain positive-definite continuous functions on $G\times G$ with an ordinary group $G$ in Olshanski's theory. Our naive idea here is to translate positive-definite continuous functions on groups into positive linear functionals on corresponding $C^*$-algebra.

It is fairly elementary to check that the above $\varphi$ is binormal and satisfies that
\begin{equation}\label{eq:sp}
(\kappa(x)\otimes1)\cdot\varphi=(1\otimes x)\cdot\varphi
\end{equation}
for every $x\in M_\tau$. Here is an abstract definition of spherical functions:

\begin{definition}
A binormal state $\varphi$ on $M\otimes_\mathrm{bin} M$ is called a \emph{spherical function} if $\varphi$ satisfies Equation \eqref{eq:sp}.
\end{definition}

The following lemma gives naturality of quantized characters from the viewpoint of spherical representation theory.
\begin{lemma}\label{lem:pikopiko}
For any spherical function $\varphi$, the state $\varphi_L$ on $M$ becomes a quantized character of $G$. Moreover, $(T_L,\mathcal{H},\xi)$ gives the GNS-triple associated with $\varphi_L$ if $\varphi$ is associated with a spherical representation $(T,\mathcal{H},\xi)$.
\end{lemma}
\begin{proof}
Since $\varphi_L$ is normal, it suffices to show that $\varphi_L$ is a $\tau$-KMS state. Since $\tau_{-i}=\kappa^2$ and $\tau_{\mathrm{i}/2}(y)^*=\tau_{-\mathrm{i}/2}(y^*)$, we have
\begin{align*}
\varphi_L(x\tau_{-\mathrm{i}}(y))
&=[(\kappa^2(y)\otimes1)\cdot\varphi](x\otimes1)\\
&=[(1\otimes\kappa(y))\cdot\varphi](x\otimes1)\\
&=\overline{[(1\otimes\kappa^{-1}(y^*))\cdot\varphi](x^*\otimes1)}\\
&=\overline{[(y^*\otimes1)\cdot\varphi](x^*\otimes1)}\\
&=\varphi_L(yx)
\end{align*}
for any $x\in M$ and $y\in M_{\tau}$. If $\varphi$ is associated with a spherical representation $(T,\mathcal{H},\xi)$, then it is clear that $\varphi_L(x)=\langle T_L(x)\xi,\xi\rangle$ for any $x\in M$. Moreover, $\xi$ is cyclic for $T_L(M)$ (see the proof of Lemma \ref{lem:CW}). Thus, $(T_L,\mathcal{H},\xi)$ is the GNS-triple associated with $\varphi$.
\end{proof}

\begin{remark}
In a similar way to Lemma \ref{lem:pikopiko}, we can prove that $\varphi_R(x\kappa^{-2}(y))=\varphi_R(yx)$ for any $x\in M$ and $y\in M_\tau$. Since $\kappa^{-2}=\tau_{i}$, we conclude that $\varphi_R$ is $\tau$-KMS state on $M$ with the inverse temperature $1$.
\end{remark}

We are in a position to give a precise relationship between spherical functions and quantized characters. The next theorem can be understood as a quantum analog of one of the key facts in Olshanski's spherical representation theory (see [13, Section 24]). %We may also regard this result as additions to Woronowicz's work on the purification map \cite{Woronowicz72}.

\begin{theorem}\label{thm:ul}
The following four assertions hold true:
\begin{enumerate}
\item The correspondence $(T,\mathcal{H},\xi)\mapsto \varphi_{(T,\mathcal{H},\xi)}$ gives a bijection from the unitarily equivalent classes of spherical representations to the spherical functions.
\item The correspondence $\chi\mapsto (T_\chi,\mathcal{H}_\chi,\xi_\chi)$ gives a bijection from the spherical functions to the unitarily equivalent classes of spherical representations.
\item By (1), (2), the correspondence $\chi\mapsto\varphi_{(T_\chi,\mathcal{H}_\chi,\xi_\chi)}$ gives a bijection from the spherical functions to the quantized characters. Moreover, the inverse correspondence is given as $\varphi\mapsto\varphi_L$.
\item Under the bijection in (2) (resp. in (1)), extreme quantized characters (resp. extreme spherical functions) correspond to irreducible spherical representations.
\end{enumerate}
\end{theorem}
\begin{proof}
To prove (1), (2) and (3), it suffices to show the following three claims:
\begin{description}
\item[(a)] Any spherical representation $(T,\mathcal{H},\xi)$ is unitarily equivalent to $(T_\chi,\mathcal{H}_\chi,\xi_\chi)$, where $\chi:=\varphi_L$ and $\varphi:=\varphi_{(T,\mathcal{H},\xi)}$.
\item[(b)] Any spherical function $\varphi$ is equal to $\varphi_{(T,\mathcal{H},\xi)}$, where $(T,\mathcal{H},\xi):=(T_\chi,\mathcal{H}_{\chi},\xi_{\chi})$ and $\chi:=\varphi_L$.
\item[(c)] Any quantized character $\chi$ of $G$ is equal to $\varphi_L$, where $\varphi:=\varphi_{(T,\mathcal{H},\xi)}$ and $(T,\mathcal{H},\xi):=(T_\chi,\mathcal{H}_\chi,\xi_\chi)$.
\end{description}
Claims (a) and (c) clearly follow from Lemma \ref{lem:CW}(1), (3) and Lemma \ref{lem:pikopiko}. We prove Claim (b). Let $\varphi$ be a spherical function. By Lemma \ref{lem:pikopiko} and Lemma \ref{lem:CW}(1), we obtain a quantized character $\chi:=\varphi_L$ of $G$ and the spherical representation $(T_\chi,\mathcal{H}_\chi,\xi_\chi)$. Then we have 
$\langle T_\chi(x\otimes y)\xi_\chi,\xi_\chi\rangle=\chi(x\kappa(y))=\varphi(x\kappa(y)\otimes1)=\varphi(x\otimes y)$ for any $x\in M$ and $y\in M_\tau$.
Since $M_\tau$ is $\sigma$-weakly dense in $M$ and $T_\chi$ and $\varphi$ are binormal, we obtain $\varphi_{(T_\chi,\mathcal{H}_\chi,\xi_\chi)}=\varphi$.

The forth assertion follows from Lemma \ref{lem:CW}(2).
\end{proof}

Next, we will give two propositions, which can be regarded as quantum analogs of other key facts in Olshanski's spherical representation theory. For a $*$-representation $(T,\mathcal{H})$ of $M\otimes_\mathrm{bin} M$ we define
\[\mathcal{H}^0:=\{\eta\in\mathcal{H}\mid T(1\otimes x)\eta=T(\kappa(x)\otimes 1)\eta\text{ for any }x\in M_{\tau}\}.\]
The first one is rather easy to prove. Hence we leave it to the reader.

\begin{proposition}
Let $(T,\mathcal{H},\xi)$ be a spherical representation. Then $(T,\mathcal{H})$ is irreducible if $\dim\mathcal{H}^0=1$.
\end{proposition}
%\begin{proof}
%Remark that $\mathcal{H}^0$ is closed subspace. We denote by $p$ the orthogonal projection onto $\mathcal{H}^0$. Assume that $\dim\mathcal{H}^0=1$, i.e., $\mathcal{H}^0=\mathbb{C}\xi$. For a $T(M\otimes_\mathrm{bin} M)$-invariant closed subspace $\mathcal{K}$ we have either $p\mathcal{K}=\{0\}$ or $p\mathcal{K}\neq\{0\}$. If $p\mathcal{K}=\{0\}$, then we have $\langle T(X)\xi,\eta\rangle=\langle\xi, T(X^*)\eta\rangle=0$ for any $\eta\in\mathcal{K}$ and $X\in M\otimes_\mathrm{bin}  M$. Thus $\eta=0$, i.e., $\mathcal{K}=\{0\}$, since $\xi$ is cyclic. If $p\mathcal{K}\neq\{0\}$, then $\xi\in\mathcal{K}$ and $\mathcal{K}=\mathcal{H}$ since $\xi$ is cyclic again. Therefore, $(T,\mathcal{H})$ is irreducible.
%\end{proof}

The usual proof of the following fact in the original setting uses the notion of Gelfand pairs crucially, but our proof below uses the Connes Radon--Nikodym theorem instead. It seems an interesting question to formulate a quantum analog of Gelfand pairs (in this context).
\begin{proposition}
Let $(T,\mathcal{H})$ be a $*$-representation of $M\otimes_\mathrm{bin} M$ such that $T_L$ and $T_R$ are normal. Then $\dim\mathcal{H}^0\leq1$ if $(T,\mathcal{H})$ is irreducible.
\end{proposition}
\begin{proof}
Assume that $(T,\mathcal{H})$ is irreducible and $\mathcal{H}^0\neq\{0\}$. Let $\xi_i \in\mathcal{H}^0 (i=1,2)$ be unit vectors. It suffices to show that $\xi_2$ is proportional to $\xi_1$. Since $(T,\mathcal{H})$ is irreducible, $\xi_i$ is cyclic for $T(M\otimes_\mathrm{bin} M)$. Thus $(T,\mathcal{H},\xi_i)$ is a spherical representation, and hence, by Lemma \ref{lem:pikopiko}, we obtain a quantized character $\chi_i$ of $G$ and its GNS-triple $(T_L,\mathcal{H},\xi_i)$. By Lemma \ref{lem:CW}(2), (3), the representation $(T_L,\mathcal{H})$ is factorial.

Recall that the state $\hat\chi_i$ on $T_L(M)$ given by $\hat\chi_i(T_L(x))=\chi_i(x)$ for any $x\in M$ is faithful normal and its modular automorphism group $\{\sigma^{\hat\chi_i}_t\}_{t\in\mathbb{R}}$ is given by $\sigma^{\hat\chi_i}_t(T_L(x))=T_L(\tau_t(x))$ for any $x\in M$ and $t\in\mathbb{R}$. Namely, $\sigma^{\hat\chi_1}_t=\sigma^{\hat\chi_2}_t$ for any $t\in\mathbb{R}$. Then, by \cite[Theorem VI\hspace{-.1em}I\hspace{-.1em}I.3.3(d)]{Takesaki:book}, the Connes Radon--Nikodym cocycle $\{(D\hat\chi_1:D\hat\chi_2)_t\}_{t\in\mathbb{R}}$ of $\hat\chi_1$ with respect to $\hat\chi_2$ must fall into the center of $T_L(M)$. Since $(T_L,\mathcal{H})$ is a factor representation, the center of $T_L(M)$ must be trivial, and thus $\{(D\hat\chi_1:D\hat\chi_2)_t\}_{t\in\mathbb{R}}$ is a one-parameter group of scalar unitary operators, i.e., $(D\omega_1:D\omega_2)_t=\lambda^{\mathrm{i}t}$ for some $\lambda>0$. Using the well-known uniqueness result for Connes Radon--Nikodym cocycle, we can prove that $\hat\chi_1=\lambda\hat\chi_2$. Evaluating this equation at 1 we obtain that $\lambda=1$, that is, $\hat\chi_1=\hat\chi_2$ and hence $\chi_1=\chi_2$. By Theorem \ref{thm:ul}, $(T,\mathcal{H},\xi_1)$ and $(T,\mathcal{H},\xi_2)$ are unitarily equivalent. Since $(T,\mathcal{H})$ is irreducible, any unitary intertwiner from $(T,\mathcal{H},\xi_1)$ to $(T,\mathcal{H},\xi_2)$ must be a scalar operator, that is, $\xi_2$ is proportional to $\xi_1$.

\end{proof}

\section*{Acknowledgment}
The author gratefully acknowledges the passionate guidance and continuous encouragement from his supervisor, Professor Yoshimichi Ueda. The author also thanks Professor Yuki Arano and Professor Grigori Olshanski for useful comments and discussions. The author seriously considered spherical representation theory thanks to Professor Olshanski's comments for the first version of this paper. This work was supported by JSPS Research Fellowship for Young Scientists (KAKENHI Grant Number JP 19J21098).
}


\begin{thebibliography}{99}
%\bibitem{AAR} G. E. Andrews, R. Askey, R. Roy, \emph{Special Functions}, Encyclopedia of Mathematics and its Applications \textbf{71}, Cambridge Univ. Press, 1999.
\bibitem{Blackadar} B. Blackadar, \emph{Operator Algebras : Theory of $C^*$-Algebras and von Neumann Algebras}, Encyclopedia of Mathematical Sciences \textbf{122}, Springer-Verlag Berlin Heidelberg, 2006.
\bibitem{BratteliRobinson1} O. Bratteli, D. W. Robinson, \emph{Operator Algebras and Quantum Statistical Mechanics 1. Equilibrium states. Models in quantum statistical mechanics. Second edition}, Texts and Monographs in Physics, Springer-Verlag, Berlin, Heidelberg, 1997.
\bibitem{BratteliRobinson2} O. Bratteli, D. W. Robinson, \emph{Operator Algebras and Quantum Statistical Mechanics 2. Equilibrium states. Models in quantum statistical mechanics. Second edition}, Texts and Monographs in Physics, Springer-Verlag, Berlin, Heidelberg, 1997.
\bibitem{BorodinOlshanski} A. Borodin, G. Olshanski, Harmonic analysis on the infinite-dimensional unitary group and determinantal point processes, \emph{Ann. of Math.} (2) \textbf{161} (2005), no. 3, 1319--1422
%\bibitem{Boyer83} R. P. Boyer, Infinite traces of AF-algebras and characters of $U(\infty)$, \emph{J. Operator Theory} \textbf{9} (1983), 205--236.
%\bibitem{Caspers} M. Caspers, Spherical Fourier transforms on locally compact quantum Gelfand pairs, \emph{SIGMA Symmetry Integrability Geom. Methods Appl.} \textbf{7} (2011), Paper 087, 39 pp.
\bibitem{Cuenca} C. Cuenca, Asymptotic Formulas for Macdonald Polynomials and the Boundary of the $(q,t)$-Gelfand--Tsetlin graph, \emph{SIGMA} {\bf 14} (2018), No. 001, 66pp.
\bibitem{EnomotoIzumi} T. Enomoto, M. Izumi, Indecomposable characters of infinite dimensional groups associated with operator algebras, \emph{J. Math. Soc. Japan} {\bf 68} (2016), no. 3, 1231--1270
%\bibitem{FSS} U. Franz, A. Skalski, P. M. So\l tan, Introduction to compact and discrete quantum groups. \emph{Topological quantum groups}, \emph{Banach Center Publ.} \textbf{111} (2017), 9--31.
\bibitem{Gorin12} V. Gorin, The $q$-Gelfand--Tsetlin graph, Gibbs measures and $q$-Toeplitz matrices, \emph{Adv. Math.} \textbf{229} (2012), 201--266.
%\bibitem{GorinOlshanski} V. Gorin, G. Olshanski, A quantization of the harmonic analysis on the infinite-dimensional unitary group, \emph{J. Funct. Anal.} \textbf{270} (2016), no. 1, 375--418
%\bibitem{GorinPanova} V. Gorin, G. Panova, Asymptotics of symmetric polynomials with applications to statistical mechanics and representation theory, \emph{Ann. Probab.} {\bf 43} (2015), no. 6, 3052--3132
\bibitem{Hirai} T. Hirai, E. Hirai, Positive definite class functions on a topological group and characters of factor representations, \emph{J. Math. Kyoto Univ.} \textbf{45} (2005), no. 2, 355--376
\bibitem{KilSch} A. Klimyk, K. Schm\"udgen, \emph{Quantum Groups and Their Representations}, Texts and Monographs in Physics, Springer-Verlag, Berlin, 1997.
\bibitem{MasudaNakagami} T. Masuda, Y. Nakagami, A von Neumann algebra framework for the duality of the quantum groups, \emph{Publ. Res. Inst. Math. Sci.} \textbf{30}(1994), no. 5, 799--850
%\bibitem{MeyerRoyWoronowicz} R. Meyer, S. Roy, S. L. Woronowicz, Homomorphisms of quantum groups, \emph{M\"unster J. Math.} {\bf 5} (2012), 1--24 
\bibitem{NeshveyevTuset} S. Neshveyev, L. Tuset, \emph{Compact Quantum Groups and Their Representation Categories}, Soc. Math. France., 2013.
\bibitem{NoumiYamadaMimachi} M. Noumi, H. Yamada, K. Mimachi, Finite dimensional representations of the quantum group $GL_q(n;\mathbb{C})$ and the zonal spherical functions on $U_q(n-1)\backslash U_q(n)$, \emph{Japan J. Math.} {\bf 19} (1993), no. 1, 31--80
%\bibitem{OkounkovOlshanski} A. Okounkov, G. Olshanski, Asymptotics of Jack Polynomials as the Number of Variables Goes to Infinity, \emph{Internat. Math. Res. Notices} (1998), no. 13, 641--682
\bibitem{Olshanski90} G. Olshanski, Unitary representations of infinite-dimensional pairs $(G,K)$ and the formalism of R. Howe, In \emph{Representation of Lie groups and related topics} (A. M. Vershik and D. P. Zhelobenko, eds.), Adv. Stud. Comtemp. Math. \textbf{7}, Gordon and Breach, New York, 1990, 269--468 
\bibitem{Olshanski03} G. Olshanski, The problem of harmonic analysis on the infinite-dimensional unitary group, \emph{Journal of Funct. Anal.}, \textbf{205} (2003), 464--524
%\bibitem{Pedersen} G. Pedersen, \emph{$C^*$-algebras and their Automorphism Group}, L.M.S. Monographs \textbf{14}, Academic Press Inc. 1979.
\bibitem{PodlesWoronowicz} P. Podle\'s, S. L. Woronowicz, Quantum Deformation of Lorentz Group, \emph{Commun. Math. Phys.} \textbf{130} (1990), 381--431.
\bibitem{Sato1} R. Sato, Quantized Vershik--Kerov theory and quantized central probability measures on branching graphs, \emph{Journal of Funct. Anal.}, \textbf{277} (2019), 2522--2557
\bibitem{Sato2} R. Sato, Type classification of extreme quantized characters, \emph{Ergodic Theory Dynam. Systems}, to appear.
\bibitem{StraVoic:book} S. Stratila, D. Voiculescu, \emph{Representations of AF-Algebras and of the Group U$(\infty)$}, Lecture Notes in Mathematics {\bf 486}, Springer-Verlag, Berlin-New York, 1975.
\bibitem{Takeda2} Z. Takeda, On the representations of operator algebra, \emph{Proc. Japan Acad.} \textbf{30} (1954), 299--304.
\bibitem{Takeda3} Z. Takeda, On the representations of operator algebra I\hspace{-.1em}I, \emph{T\^ohoku Math. J.} (2)\textbf{6} (1954), 212--219.
\bibitem{Takeda} Z. Takeda, Inductive limit and infinite direct product of operator algebras, \emph{T\^ohoku Math. J.} (2){\bf 7} (1955), 67--86. 
\bibitem{Takesaki} M. Takesaki, Algebraic equivalence of locally normal representations, \emph{Pacific J. Math.} \textbf{34} (1970), 807--816.
\bibitem{Takesaki:book} M. Takesaki, \emph{Theory of Operator Algebras I\hspace{-.1em}I}, Encyclopedia of Mathematical Sciences {\bf 127}, Springer, 2003.
\bibitem{Tomatsu} R. Tomatsu, A characterization of right coideals of quotient type and its application to classification of Poisson boundaries, \emph{Commun. Math. Phys.} \textbf{275} (2007), 271--296
%\bibitem{UenoTakebayashiShibukawa1} K. Ueno, T. Takebayashi, Y. Shibukawa, Gelfand-Zetlin Basis for $U_q(\mathfrak{gl}(N+1))$ Modules, \emph{Lett. Math. Phys.} {\bf 18} (1989), 215--221
%\bibitem{VK82} A. M. Vershik, S. V. Kerov, Characters and factor representations of the infinite unitary group, \emph{Sov. Math. Dokl.} \textbf{26} (1982), 570--574.
\bibitem{Voiculescu76} D. Voiculescu, Repr\'esentations factorielles de type I\hspace{-.1em}I de $U(\infty)$, \emph{J. Math. Pures Appl.} \textbf{55} (1976), 1--20.
%\bibitem{Wegge-Olsen} N. E. Wegge-Olsen, \emph{K-Theory and $C^*$-Algebras}, Oxford Science Publications, Oxford University Press, 1993.
%\bibitem{Woronowicz72} S. L. Woronowicz, On the purification of factor states, \emph{Comm. Math. Phys.} \textbf{28} (1972), 221--235.
\bibitem{Yamagami} S. Yamagami, On unitary representation theories of compact quantum groups, \emph{Commun. Math. Phys.}, \textbf{167} (1995), 509--529
\bibitem{Zelobenko} D. P. \u{Z}elobenko, Compact Lie groups and their representations, \emph{Translations of Mathematical Monographs} \textbf{40}, Amer. Math. Soc., 1973.
\end{thebibliography}
\end{document}